\documentclass[a4paper,10pt]{amsart}

\usepackage{latexsym} 
\usepackage{amsfonts} 
\usepackage{amssymb} 
\usepackage{color}
\usepackage{marginnote}
\usepackage{amscd,enumerate}
\usepackage{epsfig}
\usepackage{amsmath}
\usepackage[all]{xy}
\usepackage{bbm}
\usepackage{bm}

\font\pri=eufm10 at 11pt
\def\ppr#1{\hbox{\pri#1}}

\newcommand{\Z}{{\mathbb{Z}}}

\newcommand{\Q}{{\mathbb{Q}}}
\newcommand{\R}{{\mathbb{R}}}
\newcommand{\C}{{\mathbb{C}}}
\newcommand{\K}{{\mathbb{K}}}
\newcommand{\F}{{\mathbb{F}}}

\newtheorem{theo}{Theorem}
\newtheorem{pr}{Proposition}
\newtheorem{co}{Corollary}
\newtheorem{lm}{Lemma}
\newtheorem{re}{Remark}
\newtheorem{de}{Definition}

\def\GL{\mathop{\text{\rm GL}}}
\def\PGL{\mathop{\text{\rm PGL}}}
\def\pgl{\mathop{\text{\rm pgl}}}
\def\gl{\mathop{\ppr{gl}}}
\def\sl{\mathop{\ppr{sl}}}
\def\O{\mathop{\text{\rm O}}}
\def\o{\mathop{\ppr{o}}}

\def\diag{\mathop{\hbox{diag}}}
\def\iu{\mathbbm{i}} 
\def\a{\alpha}
\def\b{\beta}
\def\alg{\mathop{\text{\rm\bf alg}}}
\def\lie{\mathop{\text{\rm\bf Lie}}}
\font\got=eufb10
\font\gots=eufb7
\def\r{\text{\got{r}}}
\def\im{\text{\got{i}}}
\def\j{\text{\got{j}}}
\def\L{\text{\got{L}}}
\def\m{\text{\got{m}}}
\def\ms{\text{\gots{m}}}
\def\Z{\mathbb Z}
\def\p{\mathfrak{p}}
\def\aut{\mathop{\text{\rm Aut}}}
\def\inn{\mathop{\text{\rm Inn}}}
\def\der{\mathop{\text{\rm Der}}}
\def\z{\mathop{\text{Z}}}
\def\ad{\mathop{\text{ad}}}
\def\Ad{\mathop{\text{Ad}}}

\def\Max{\mathop{\text{Max}}}
\def\rad{\mathop{\text{rad}}}
\def\ann{\mathop{\text{Ann}}}

\def\ia{\mathfrak{a}}
\def\ib{\mathfrak{b}}
\def\ic{\mathfrak{c}}
\def\id{\mathfrak{d}}

\def\grp{\text{\bf Grp}}
\def\affpgl{\mathop{\bf PGL}}
\def\affgl{\mathop{\bf GL}}
\def\affgm{\mathop{\bf G_m}}
\def\affaut{\mathop{\bf Aut}}
\def\affAd{\mathop{\bf Ad}}

\title[Lorentz and Poincar\'e algebras]{Algebraic structure of the Lorentz and of the Poincar\'e Lie algebras} 

\author[P. Alberca]{Pablo Alberca Bjerregaard}
\address{P. Alberca Bjerregaard: Departamento de Matem\'atica Aplicada, Escuela T\'ecnica Superior de Ingenieros Industriales, 
Universidad de M\'alaga. 29071 M\'alaga. Spain.}
\email{pgalberca@uma.es}
\author[D. Mart\'{\i}n]{Dolores Mart\'{\i}n Barquero}
\address{D. Mart\'{\i}n Barquero: Departamento de Matem\'atica Aplicada, Escuela T\'ecnica Superior de Ingenieros Industriales, 
Universidad de M\'alaga. 29071 M\'alaga. Spain.}
\email{dmartin@uma.es}
\author[C. Mart\'{\i}n]{C\'andido Mart\'{\i}n Gonz\'alez}
\address{C. Mart\'{\i}n Gonz\'alez:  Departamento de \'Algebra Geometr\'{\i}a y Topolog\'{\i}a, Fa\-cultad de Ciencias, 
Universidad de M\'alaga, Campus de Teatinos s/n. 29071 M\'alaga. Spain.}
\email{candido@apncs.cie.uma.es}

\author[D. Ndoye]{Daouda Ndoye}
\address{D. Ndoye: D\'epartement d'Alg\`ebre, de G\'eomtrie et Application, Universit\'e Cheikh Anta Diop de Dakar, AIMS-S\'en\'egal, B.P 1418 Mbour-S\'en\'egal.}
\email{daouda87@hotmail.fr}


\begin{document}
\maketitle
\begin{abstract}
We start with the Lorentz algebra $\L=\o_{\R}(1,3)$ over the reals and find a suitable basis $B$ relative to which the structure constants are integers.
Thus we consider the $\Z$-algebra $\L_{\Z}$ which is free as a $\Z$-module and its $\Z$-basis is $B$. This allows us to define the Lorentz type algebra
$\L_K:=\L_{\Z}\otimes_{\Z} K$ over any field $K$. In a similar way, we consider Poincar\'e type algebras over any field $K$.

In this paper we study the ideal structure of Lorentz and of Poincar\'e type algebras over different fields.
It turns out that Lorentz type algebras are simple if and only if the ground field has no square root of $-1$. 
Thus, they are simple over the reals but not over the complex. 
Also, if the ground field is of characteristic $2$ then Lorentz and Poincar\'e type algebras are neither simple nor semisimple. 
We extend the study of simplicity of the Lorentz algebra to the case of a ring of scalars where we have to use the notion
of $\m$-simplicity (relative to a maximal ideal $\m$ of the ground ring of scalars).

The Lorentz type algebras over a finite field $\F_q$ where $q=p^n$ and $p$ is odd are simple 
 if and only if $n$ is odd and $p$ of the form $p=4k+3$. In case $p=2$ then the Lorentz type algebra are not simple.
 Once we know the ideal structure of the algebras, we get some information of their automorphism groups.
 For the Lorentz type algebras (except in the case of characteristic $2$) we describe the affine group scheme of automorphisms and 
 the derivation algebras. For the Poincar\'e algebras we restrict this program to the case of an algebraically closed field of
 characteristic other than $2$.
\end{abstract}

\section{Introduction and preliminary definitions}

\subsection{Category language} 
All through this paper $\Phi$ will denote an associative, commutative ring with unit and $\alg_\Phi$ the category whose objects are 
the associative, commutative and unital $\Phi$-algebras. On the other hand, $\lie_\Phi$ will denote the category of Lie $\Phi$-algebras.
We will have the ocassion to deal with (covariant) functors ${\mathcal F}\colon \alg_\Phi\to\lie_\Phi$. These functors will be called Lie algebra functors since
they take values in $\lie_\Phi$. Given two Lie algebra functors ${\mathcal F}, {\mathcal G}\colon \alg_\Phi\to\lie_\Phi$ a homomorphism $\eta\colon{\mathcal F}\to{\mathcal G}$
is a natural transformation from $\mathcal F$ to $\mathcal G$, that is, a family $\{\eta_R\}$ where:
\begin{enumerate}
 \item $R$ ranges in the class of objects of $\alg_\Phi$,
 \item $\eta_R\colon {\mathcal F}(R)\to{\mathcal G}(R)$ is a homomorphism of Lie $\Phi$-algebras.
 \item For any two objects $R$ and $S$ in $\alg_\Phi$ and any homomorphism of $\Phi$-algebras $\alpha\colon R\to S$, the following squares commute:
 \[
\xygraph{
!{<0cm,0cm>;<1cm,0cm>:<0cm,1cm>::}
!{(0,0)}*+{{\mathcal F}(R)}="p"
!{(1.5,0)}*+{{\mathcal G}(R)}="i"
!{(0,-1.2)}*+{{\mathcal F}(S)}="j"
!{(1.5,-1.2)}*+{{\mathcal G}(S).}="r"
"p":^{\eta_R}"i"
"p":_{{\mathcal F}(\alpha)}"j"
"i":^{{\mathcal G}(\alpha)}"r"
"j":_{\eta_S}"r"
}
\]\end{enumerate}
We will say that $\mathcal F$ is isomorphic to ${\mathcal G}$ if all the $\eta_R$ are isomorphisms (in this case we will use any of the notations
$\eta\colon {\mathcal F}\cong{\mathcal G}$, ${\mathcal F}\buildrel{\eta}\over{\cong}{\mathcal G}$ or
${\mathcal F}\cong{\mathcal G}$).\medskip

\begin{de}
Consider next the full subcategory $\sqrt{-1}_\Phi$ of $\alg_\Phi$ whose objects are the $\Phi$-algebras $R$ such that $\sqrt{-1}\in R$. 
Denote by ${\mathcal I}$ the inclusion functor ${\mathcal I}\colon\sqrt{-1}_\Phi\to\alg_\Phi$. 
\end{de} 

\subsection{The Lorentz functor}\label{llacer} 
The  Lorentz algebra over the reals, denoted by $\o(1,3)$, is the Lie algebra of the orthogonal Lie group $\O(1,3)$:
$$\o(1,3)=\text{Lie}(O(1,3))=\{M\in \gl\nolimits_4(\R)\colon MI_{13}+I_{13}M^t=0\},$$
where $M^t$ denotes matrix transposition of $M$ and $I_{13}=\diag(-1,1,1,1)$ (some authors take $I_{13}=\diag(1,1,1,-1)$ which is equivalent). A 
straightforward computation reveals that a generic element of $\o(1,3)$ is of the form
$$\left(
\begin{array}{cccc}
 0 & x_1 & x_2 & x_3 \\
 x_1 & 0 & x_4 & x_5 \\
 x_2 & -x_4 & 0 & x_6 \\
 x_3 & -x_5 & -x_6 & 0
\end{array}
\right)$$
and then denoting by $e_{ij}$ the elementary matrix with $1$ in the entry $(i,j)$ and $0$ elsewhere we have a basis
of $\o(1,3)$ given by
$B=\{s_{12},s_{13},s_{14},a_{23},a_{24},a_{34}\}$ where $s_{ij}:=e_{ij}+e_{ji}$ and $a_{ij}=e_{ij}-e_{ji}$.
\begin{figure}[h]
\begin{tabular}{|c|cccccc|}
\hline
$[\ ,\ ]$ & $s_{12}$ & $s_{13}$ & $s_{14}$ & $a_{23}$ & $a_{24}$ & $a_{34}$\\ 
\hline
$s_{12}$ & $0$ & $a_{2,3}$ & $a_{2,4}$ & $s_{1,3}$ & $s_{1,4}$ & $0$ \\
$s_{13}$& $-a_{23}$ & $0$ & $a_{34}$ & $-s_{12}$ & $0$ & $s_{14}$
   \\
$s_{14}$& $-a_{24}$ & $-a_{34}$ & $0$ & $0$ & $-s_{12}$ &
   $-s_{13}$ \\
$a_{23}$& $-s_{13}$ & $s_{12}$ & $0$ & $0$ & $-a_{34}$ & $a_{24}$
   \\
$a_{24}$& $-s_{14}$ & $0$ & $s_{12}$ & $a_{34}$ & $0$ & $-a_{23}$
   \\
$a_{34}$& $0$ & $-s_{14}$ & $s_{13}$ & $-a_{24}$ & $a_{23}$ & $0$\\
\hline
 \end{tabular}
 \caption{Multiplication table of $\o(1,3)$.}\label{ttaudi}
\end{figure}
Relative to this basis the structure constants are $0$, $1$ or $-1$. Thus we can construct the $\Z$-algebra
$\L_\Z:=\Z s_{12}\oplus \Z s_{13}\oplus \Z s_{14}\oplus \Z a_{23}\oplus \Z a_{24}\oplus \Z a_{34}$ whose multiplication
table is given in Figure \ref{ttaudi}. Fix now an associative, commutative and unital ring $\Phi$ and consider the category
$\alg_\Phi$ defined above.

Then for any object $R$ in $\alg_\Phi$ we may define the Lorentz type algebra 
$\L_R:=\L_\Z\otimes_\Z R$.
This is nothing but the free $R$-module with basis $s_{12}$, $s_{13}$, $s_{14}$, $a_{23}$, $a_{24}$ and $a_{34}$,
enriched with an $R$-algebra structure by the multiplication table as in Figure \ref{ttaudi}. As a free $R$-module we have
$$\dim \L_R=6.$$
Of course if we take $R=\R$ then $\L_R\cong\o(1,3)$, the Lorentz algebra. If $R=\C$ then $\L_R$ is the complexified Lorentz algebra.
  If $R$ and $S$ are objects in $\alg_\Phi$ and $f\colon R\to S$ a $\Phi$-algebras homomorphism,
then we may define a Lie $\Phi$-algebras homomorphism $\L_f\colon \L_R\to \L_S$ in an obvious way. 
Thus we have defined a covariant functor $\L\colon\alg_\Phi\to \lie_\Phi$ (where 
$\lie_\Phi$ is the category of Lie $\Phi$-algebras).
\medskip

Let $\O(n)$ be the orthogonal Lie group over the reals: the group of all matrices $M$ in $\GL_n(\R)$ such that $MM^t=1_n$.
Then, its Lie algebra $\o(n)$ consists of all matrices $M$ in $\gl_n(\R)$ such that $M+M^t=0$. This is generated (as a vector space) by the matrices
$e_{ij}-e_{ji}$ where $i<j$ with $i,j\in\{1,\ldots,n\}$ and the structure constants relative to the basis of these elements are again $0$ or $\pm 1$.
Thus, we can consider as before the $\Z$-algebra $\o(n;\Z):=\oplus_{i<j}\Z(e_{ij}-e_{ji})$.
Fix as before a ring $\Phi$ and then, for any algebra $R$ in $\alg_\Phi$ we may define the scalar extension 
$\o(n;R):=\o(n;\Z)\otimes_{\Z} R$. So, this is the Lie $R$-algebra with basis $e_{ij}-e_{ji}$ as before and
multiplication table as the one for $\o(n)$ in the corresponding basis. 

Thus we have $\dim_R(\o(n;R))=n(n-1)/2$ and 
we have again a functor $$\o(n)\colon\alg\nolimits_\Phi\to\lie\nolimits_\Phi$$ such that $R\mapsto\o(n;R)$. If $f\colon R\to S$ is a homomorphism
of algebras in $\alg_\Phi$ then we will denote by $\o(n;f)\colon\o(n;R)\to\o(n;S)$ the homomorphism of Lie algebras $\o(n;f):=1\otimes f$. Along this work, the alternative notation $\o_n(R)$ (meaning $\o(n;R)$) will be used eventually. 
\begin{re}\rm
 If $\Phi$ is a ring agreeing with its $2$-torsion, that is, $1+1=0$, then for any $\Phi$-algebra $R$
 in $\alg_\Phi$, the Lie algebra $\o(4;R)$ agrees with the Lorentz type Lie algebra $\L_R$. In particular
 this is the case for a field $\K$ of characteristic two: $\L_\K=\o(4;\K)$. A more general result is the following.
 \end{re}
 
\begin{lm}\label{one}
For any $\Phi$, the functors $\L\circ {\mathcal I}$ and $\o(4)\circ{\mathcal I}:\sqrt{-1}_\Phi\to\lie_\Phi$ are isomorphic. 
More precisely (i) for any algebra $R$ in $\alg_\Phi$ such that the equation $x^2+1=0$ has a solution in $R$, 
there is an isomorphism $\eta_R\colon \L_R\cong\o(4;R)$; (ii) If $f\colon R\to S$ is a homomorphism of $\Phi$-algebras and $\sqrt{-1}\in R$,
the following diagram commutes:
 \[
\xygraph{
!{<0cm,0cm>;<1cm,0cm>:<0cm,1cm>::}
!{(0,0)}*+{\L_R}="p"
!{(1.5,0)}*+{\o(4;R)}="i"
!{(0,-1.2)}*+{\L_S}="j"
!{(1.5,-1.2)}*+{\o(4;S).}="r"
"p":^{\eta_R}"i"
"p":_{\L_f}"j"
"i":^{\o(4;f)}"r"
"j":_{\eta_S}"r"
}
\]
\end{lm}
\begin{proof} Take $\iu\in R$ such that $\iu^2=-1$. Starting from the standard basis $B$ of $\L_R$, we define a new basis $C=\{a'_{ij}\colon i,j\in\{1,2,3,4\}, i<j\}$ where
$a'_{12}:=\iu s_{12}$, $a'_{13}:=\iu s_{13}$, $a'_{14}:=\iu s_{14}$ and $a'_{ij}:=a_{ij}$ for the remaining elements.
Then the isomorphism $\L_R\to\o(4;R)$ is the induced by $a'_{ij}\mapsto e_{ij}-e_{ji}$ for $i<j$. On the other hand, the commutativity of the
square above is straightforward.\end{proof}

For any object $R$ of $\alg_\Phi$, the Lie algebra  of the linear special group, $\sl_2(R)$, 
 is defined by $\sl_2(R)=\{A\in \gl_2(R)\colon \hbox{Tr}(A)=0\}$, where $\hbox{Tr}$ denotes the matrix trace. 
The system $\{h:=e_{11}-e_{22}, e:=e_{12}, f:=e_{21}\}$, where $e_{ij}$ is the elementary matrix with $1$ in the position $(i,j)$ 
and $0$ in the others, is a basis of $\sl_2(R)$ and their elements satisfy the following identities:
\begin{equation}~\label{idsl}
 [h,f]=-2f,\ 
 [h,e]=2e,\ 
 [e,f]=h.
\end{equation}

 Consider now algebras $R$ and $S$ in the category $\alg_\Phi$ such that $R$ is a subalgebra of $S$. Denote by $\sl_2(S)$ the Lie algebra of
$2\times 2$ matrices with entries in $S$ of zero trace. 
Any $\alpha\in\aut_R(S)$ induces
an automorphism $\hat\a\in\aut_R(\sl_2(S))$ by applying $\a$ componentwise. Also for any $P\in\GL_2(R)$ the map $M\to PMP^{-1}$ gives an
automorphism of $\sl_2(S)$ which is denoted by $\Ad(P)$. More generally, for any $P$ in a linear algebraic group $G$, the adjoint action of $G$ on its Lie algebra $\mathfrak g$ will be denoted $\Ad\colon G\to\aut({\mathfrak g})$, so that for any $M\in\mathfrak g$ we have $\Ad(P)M:=PMP^{-1}$.

\begin{lm}\label{plafter}
 Under the condition in the above paragraph if $\Ad(P)=\hat\a$, then $\a=1$ and consequently $\Ad(P)=1$.
\end{lm}
\begin{proof}
 We know that $PMP^{-1}=\hat\a(M)$ for any $M\in\sl_2(S)$, in particular since $R\subset S$ we may take $M\in\sl_2(R)$ and so
 $PMP^{-1}=M$ hence $PM=MP$ for any $M\in\sl_2(R)$. This implies $P=k\ \hbox{id}$ for some invertible $k\in R$. Thus $\Ad(P)=1$
 which implies $\hat\a=1$.
\end{proof}

\begin{lm}\label{etiq}
 Let $\K$ be a field of characteristic not two, then if $\beta\colon\o(4;\K)\to\K^4$ is a $\K$-linear map such that 
 $\beta([M,M'])=\beta(M)M'-\beta(M')M$ for any $M,M'\in\o(4;\K)$,
 there is a unique $v\in\K^4$ such that $\beta(M)=vM$ for any $M\in\o(4;\K)$.
\end{lm}
\begin{proof} 
This is a cohomological result (a version of Whitehead's lemma). It is well-known in characteristic zero and in prime characteristic may be seen
as a consequence of \cite[Theorem 1]{dzh}. However we include here a selfcontained proof. 
Let $S_1$ denote the vector space of all linear 
maps $\beta\colon\o(4;\K)\to\K^4$  such that $\beta([M,M'])=\beta(M)M'-\beta(M')M$ for any $M,M'\in\o(4;\K)$.
Denote by $S_2$ the vector space of all maps $\gamma\colon \o(4;\K)\to\K^4$ such that there is a fixed $v\in\K^4$ with
$\gamma(M)=vM$ for any $M$. It is straightforward to prove that $S_2$ is a subspace of $S_1$ and that $S_2$ has dimension $4$.
Thus, all we need to prove is that $\dim(S_1)=4$.
The best way to proceed for getting convinced of this fact is to fix a basis of $\o(4;\K)$ and to consider the coordinate map
$\chi\colon \o(4;\K)\to \K^6$ such that the components of $\chi(M)$ are the coordinates of $M$ relative to the fixed basis. Since $\beta$ is linear
there is a $6\times 4$ matrix $L$ with entries in $\K$ such that $\beta(M)=\chi(M)L$. Now, one can see that the conditions 
$\beta([M,M'])=\beta(M)M'-\beta(M')M$ for any $M,M'\in\o(4;\K)$ imply that there are only $4$ free (independent) parameters in $L$. Indeed, $L$
is of the form:
$$L=\begin{pmatrix}  a & b & 0 & 0\cr
   c & 0 & b & 0\cr
   g & 0 & 0 & b\cr
   0 & c & -a & 0\cr
   0 & g & 0 & -a\cr
   0 & 0 & g & -c\cr
\end{pmatrix}.$$

Finally the $v$ whose existence has been proved is unique since any $w\in\K^4$ such that $w\o(4;\K)=0$ must be zero.
\end{proof}


\section{Simplicity results}

We would like to study under what conditions the Lorentz functor $\L\colon\alg_\Phi\to\lie_\Phi$ produces simple Lie algebras.
The hidden motivation for this study is that when $\L_R$ is not simple, under suitable conditions, we can decompose $\L_R$ as
a certain direct sum of two ideals. These ideals are very special and our expectation on them is that the automorphisms of
$\L_R$ either fix the ideals or swap them. Thus, a knowledgement of the ideal structure of $\L_R$ immediately produces 
information about the affine group scheme $R\mapsto \text{\bf aut}(\L_R)$. 

To shorten the notations, we write $b_1:=a_{12}$, $b_2:=a_{13}$, $b_3:=a_{14}$, $b_4:=s_{23}$, $b_5:=s_{24}$, $b_6:=s_{34}$ so that the basis
$B$ of $\L_R$ is now $B=\{b_i\}_1^6$ and has the multiplication table given in Figure \ref{audia4}. Also for any $R$ in $\alg_\Phi$ we will
denote by $\hbox{Max}(R)$ the maximal spectrum of $R$ (the set of maximal ideals of $R$).

In this section we study the simplicity of Lorentz type algebras $\L_R$ where $R$ is an algebra in $\alg_\Phi$.
Since any ideal $I$ of $R$ induces trivially an ideal $I\L_R$ of $\L_R$ we will pay no attention to this class of ideals.
The best way to rule out such ideals is to rule out the ideals of the kind $\m\L_R$, where $\m\in\Max(R)$ is a maximal ideal of $R$
(since any proper ideal of $R$ is contained in some maximal one). The ideals of the kind $\m\L_R$ will be termed $\m$-null ideals 
(we will define them formally later). Other class of ideals, that we shall exclude of our study, are the $\m$-total ideals (the ideals
which agree with $\L_R/\m\L_R$ when passing to the quotient).

We start by considering an algebra $R$ in $\alg_\Phi$ and an $R$-module $M$. For any maximal ideal $\m\in\Max(R)$ we may consider the epimorphism
$\phi\colon M\to M\otimes_R \K=:M_\K$ where $\K$ is the field $\K:=R/\m$. Then $M_\K$ is a vector space over $\K$ and $\ker\phi=\m M$ (see \cite[Lemma 5, p.215]{Spanier}), so that
$M/\m M\cong M_\K$.
\begin{de}
 A collection of elements $m_1,\ldots,m_n\in M$ is said to be $\m$-free if for any $r_1,\ldots,r_n\in R$ the equality
 $\sum_i r_im_i\in\m M$ implies $r_i\in\m$ for all $i$. An $R$-module $M$ is said to have an $\m$-free part of cardinal $n$ is
 there is an $\m$-free subset $\{m_1,\ldots,m_n\}\subset M$.
\end{de}
If $R$ happens to be a field, then $\m=0$ and a set is $\m$-free if and only if it is linearly independent.
In general, a set $\{m_1,\ldots,m_n\}\subset M$ is $\m$-free if and only if the set of equivalence classes 
$\{\overline{m_1},\ldots,\overline{m_n}\}\subset M/\m M$ is a linearly independent subset
of the $\K$-vector space $M/\m M$.

Fix a ring $\Phi$ and a Lie $\Phi$-algebra $V$.
For any associative commutative and unital $\Phi$-algebra $R$ denote by $V_R$ the scalar extension $V_R:=V\otimes_\Phi R$.
If $S$ is another algebra in $\alg_\Phi$ and $S$ is an $R$-algebra we may consider also the scalar extension $V_S:=V\otimes_\Phi S$.
The reader can easily check the existence of an isomorphism $$V_R\otimes_R S\cong V_S$$ such that $(v\otimes_\Phi r)\otimes_R s\mapsto v\otimes_\Phi rs$. 

Furthermore, if $\m\in\Max(R)$ is a maximal ideal and $I$ a submodule of $V_R$  with an $\m$-free part of cardinal $n$,
then its image $\phi(I)$ under the canonical epimorphism $\phi\colon V_R\to V_R\otimes_R\K$ (where $\K=R/\m$) contains a linearly
independent set of cardinal $n$ hence $\dim_\K\phi(I)\ge n$.

\begin{de}
 An ideal $I\triangleleft V_R$ such that its image under the epimorphism $\phi\colon V_R\to V_\K$ (as above) is the whole $V_\K$ is said
 to be $\m$-total. An ideal $I\triangleleft V_R$ such that $\phi(I)=0$ (equivalently $I\subset \m V_R$) is said to be $\m$-null.
 The algebra $V_R$ is said to be $\m$-simple if its unique ideals are the $\m$-null and the $\m$-total ones.
\end{de}

Again, when $R$ is a field $\m=0$ and so $\L_R$ is $\m$-simple if and only if it is simple in the usual sense.
Thus, in our study on the simplicity of Lorentz type algebras, we will replace simplicity with $\m$-simplicity.

As an example of $\m$-total ideal consider the Lorentz algebra $\L_\Z$, the ideal $\m:=3\Z$ of $\Z$ and define 
$I:=\Z(b_1+b_6)+\Z(b_1-b_6)+\Z(b_2+b_5)+\Z(b_2-b_5)+\Z(b_3+b_4)+\Z(b_3-b_4)$. This is an ideal of $\L_\Z$.
It is proper since $b_1\not\in I$ but its image in $\L_\K$ (where $\K=\Z/3\Z$) is the whole algebra.
The reader can check that $2\L_\Z$ is also an $\m$-total ideal of $\L_\Z$ and that $3\L_\Z$ is an $\m$-null ideal of $\L_\Z$.

In special cases, it is easy to describe those ideals of $V_R$ which are $\m$-total for any $\m\in\hbox{Max}(R)$:
\begin{pr}\label{addedone}
 Assume that $V$ is a perfect Lie algebra over $\Phi$, and $R$ in $\alg_\Phi$ is artinian. 
 Then $I\triangleleft V_R$ is $\m$-total for any $\m\in\Max(R)$
 if and only if $I=V_R$.
\end{pr}
\begin{proof}
 For any $\m\in\Max(R)$ we have $V_R=I+\m V_R$. Since $V=[V,V]$, the same holds for $V_R$. If we take two ideals $\m_1,\m_2\in\Max(R)$,
 we have $V_R=I+\m_1 V_R=I+\m_2 V_R$ by the uniform $\m$-totality of $I$ when $\m\in\Max(R)$. Thus $V_R=[V_R,V_R]=I+\m_1\m_2 V_R$. 
Consequently, for any finite subset $\{\m_1,\ldots,\m_n\}\subset\Max(R)$ we have $V_R=I+\m_1\cdots\m_n V_R$. 
Since $R$ is artinian it has only a finite number of maximal ideals (see \cite[Proposition 8.3, p. 89]{Atiyah}).
Thus $V_R=I+\rad(R)V_R$ (where $\rad(\cdot)$ denotes Jacobson's radical)
and since $V$ is perfect we get $V_R=I+\rad(R)^kV_R$ for any positive integer $k$.
Also the artinian character of $R$ implies that $\rad(R)$ is nilpotent (take into account also that every prime ideal is maximal and so the Jacobson radical agrees 
with the nilradical, \cite[p.89]{Atiyah}). So, for some $k$ one has $\rad(R)^k=0$ implying $V_R=I$.
 \end{proof}

\begin{re}\label{moreno}\rm
 It is standard result that if $A$ is a free $R$-algebra and $\{I_\a\}$ a collection of ideals of the ring of scalars $R$, then
$(\cap_\a I_\a)A=\cap_\a(I_\a A)$.
It is easy to see that if an ideal $I\triangleleft V_R$ is $\m$-null for any $\m\in\Max(R)$, then 
$I\subset \underset{\ms}\cap (\m V_R)=(\underset{\ms}\cap \m) V_R=\rad(R)V_R$. Thus in the case in which the Jacobson radical of $R$ is null, the unique
ideal which is $\m$-null for every $\m\in\Max(R)$ is the $0$ ideal.
 \end{re}
 
As a consequence of Proposition \ref{addedone} and of the previous paragraph we can state:

 \begin{theo}
Assume as before that $V$ is a perfect Lie algebra, $R$ is an artinian algebra  in $\alg_\Phi$ with $\rad(R)=0$, and 
$V_R$ is free as an $R$-module. Then $V_R$ is $\m$-simple for any $\m\in\Max(R)$ if and only if any proper nonzero ideal $I$ of $V_R$ is of the  form 
$I=\im V_R$ where $\im$ is a product of maximal ideals $\im=\m_{i_1}\cdots\m_{i_k}$, $\m_{i_q}\in\Max(R)$. Furthermore, $V_R$ splits in the form 
$V_R=I\oplus J$  where $J=\j V_R$ and $\j=\m_{j_1}\cdots \m_{j_q}$ where each factor $\m_{j_r}\in\Max(R)$, 
and $\Max(R)$ is the disjoint union of $\{\m_{i_1},\cdots,\m_{i_k}\}$ and 
$\{\m_{j_1},\cdots, \m_{j_q}\}$.
 \end{theo}
\begin{proof}
 Take $I\triangleleft V_R$ which is nozero and proper.
 Since $V_R$ is $\m$-simple for any $\m$, then it may not happen that $I$ is $\m$-total for any $\m$ (see Proposition \ref{addedone}). 
 But being $I\ne 0$ it may not happen that $I$ is $\m$-null for every $\m\in\Max(R)$ (see Remark \ref{moreno}). So we may decompose $\Max(R)$ as a disjoint union
 $\Max(R)=S_1\dot\cup S_2$ where $S_1:=\{\m\in\Max(R)\colon I \text{ is $\m$-null }\}$ and
 $S_2:=\{\m\in\Max(R)\colon I \text{ is $\m$-total }\}$. Of course $\m\in S_1$ if and only if $I\subset\m V_R$ and $\m\in S_2$ if and only
 if $V_R=I+\m V_R$. Also recall that $\Max(R)$ is finite. Now, for any $\m,\m'\in S_2$ with $\m\ne\m'$ we have   
 $V_R=I+\m V_R=I+\m'V_R$ and by the perfection of $V$ we have $V_R=I+\m\m'V_R$. So we conclude that 
 $$V_R=I+\prod_{\m\in S_2} \m V_R.$$
 If $\j=\prod_{\m\in S_2}\m=\cap_{\m\in S_2}\m$, then $V_R=I+\j V_R$. On the other hand, $I\subset\cap_{\m\in S_1}(\m V_R)=
 (\cap_{\m\in S_1}\m) V_R=\im V_R$ where  $\im:=\cap_{\m\in S_1}\m=\prod_{\m\in S_1}\m$.
 Then $I\cap \j V_R\subset (\cap_{\m\in S_1}\m)V_R\cap (\cap_{\m\in S_2}\m)V_R=\rad(R)V_R=0$.
 Thus $V_R=I\oplus\j V_R$. Next we prove that $I=\im V_R$ (we already have $I\subset\im V_R$). Take $z\in\im V_R$, since
 $V_R=I\oplus \j V_R$ we have $z=i+w$ where $i\in I$ and $w\in\j V_R$. Then $w=z-i\in\j V_R\cap\im V_R=0$. Thus $z=i\in I$ and taking $J=\j V_R$ the proof is complete.\end{proof}

\begin{lm}\label{lematres}
 For $R$ in $\alg_\Phi$ such that $\frac{1}{2}\in R$, we consider the Lie $R$-algebra $\sl_2(R)$. Then any nonzero proper ideal of $\sl_2(R)$ is 
 of the form $\mathfrak{i}\sl_2(R)$ for a nonzero proper ideal $\mathfrak{i}$ of $R$. Consequently, any nonzero proper ideal of $\sl_2(R)$ is
 $\m$-null  for some maximal ideal $\m\in\Max(R)$. 
\end{lm}
\begin{proof}
 Take $0\ne I\triangleleft \sl_2(R)$. We consider the usual basis $\{h,e,f\}$ of the free $R$-module $\sl_2(R)$ such that $[h,e]=2e$, $[h,f]=-2f$
 and $[e,f]=h$. It is easy to check that for any scalar $x\in R$ we have $xh\in I\Leftrightarrow xe\in I \Leftrightarrow xf\in I$.
 Then we may define $\mathfrak{i}:=\{x\in R\colon xh\in I\}$ which is an ideal of $R$ and $I=\mathfrak{i}\sl_2(R)=\sl_2(\mathfrak{i})$.
 Now if $\mathfrak{i}=R$, then $I=\sl_2(R)$ contradicting the fact that $I$ is proper. Thus $\mathfrak{i}\ne R$ and there is a maximal ideal $\m\in\Max(R)$
 such that $\mathfrak{i}\subset\m$. Consequently $I\subset \m\sl_2(R)$ hence $I$ is $\m$-null. 
\end{proof}

\begin{re}\rm
Take as before $\m\in\Max(R)$ where $R$ is in $\alg_\Phi$ and let $I\triangleleft \L_R$ be an ideal of the Lorentz type algebra $\L_R$.
Then $I$ is not $\m$-null if and only if it has an $\m$-free subset of cardinal $\ge 1$. 
\end{re}

\begin{theo}\label{cuatrocinco} 
Let $V$ be a Lie $\Phi$-algebra which is a free $\Phi$-module of dimension $6$ and satisfies $[V,V]=V$.
Take $R$ to be an algebra in $\alg_\Phi$ and $\m\in\Max(R)$. Denote by $\K:=R/\m$ the residue field. 
Let $\phi\colon V_R\to V_\K:=V_R\otimes_R \K$ be the canonical
epimorphism $x\mapsto x\otimes \bar 1$ whose kernel is $\m V_R$. 
Assume that $I$ is an $R$-submodule with an $\m$-free part of cardinal $4$ or $5$. 
Then either $\phi(I)=V_\K$ or $I$ is not an ideal of $V_R$.
In particular, if $V$ is a Lie algebra over a field $\K$ and $[V,V]=V$ with $\dim V=6$, then $V$ has not ideals of dimension
$4$ or $5$. 
\end{theo}
\begin{proof}
First we prove the particular case. So we take $V$ to be a perfect Lie algebra over  a field $\K$  and prove that $V$ has no ideal of dimension $5$ or $4$ . 
Assume that $I$ is a $5$-dimensional ideal, then $V/I$ has dimension $1$ hence it is abelian. Thus $[V,V]\subset I$ but since $[V,V]=V$, we conclude 
that $V=I$ a contradiction. Next we prove that $V$ has no ideal of dimension $4$. If $J$ is such an ideal, there is a maximal ideal $M$ of $V$ such that 
$J\subset M$. We know that $M$ is not $5$-dimensional hence $J=M$. Thus taking into account the maximality of $M$, we conclude
that $V/M$ is a simple Lie algebra of dimension $2$. This is a contradiction because no simple Lie algebra can be $2$-dimensional.  

Now we prove the result for a general $R$ in $\alg_\Phi$. Assume that $I$ is an ideal of $V_R$ with an $\m$-free part of cardinal $4$ or $5$.
Consider the $R$-algebra $\K:=R/\m$ which is a field. Thus $\phi(I)$ is a $\K$-vector subspace of
$V_\K$ and if $\phi(I)\ne V_\K$, then 
$\dim_\K \phi(I)\in\{4,5\}$. So $\phi(I)$ is not an ideal (hence $I$ is not an ideal given the epimorphic character of $\phi$).
\end{proof}

\begin{co}
For any field $\K$, the Lorentz type algebra $\L_\K$ has no ideals of dimension $4$ or $5$.
\end{co}

Next we recall an elementary result on field theory: assume $\K$ to be a field such that $\sqrt{-1}\in\K$. Then putting $x:=\sqrt{-1}$ and $y=1$ one has
$x^2+y^2=0$ where $x,y\ne 0$. Reciprocally if $\K$ is a field such that there are nonzero elements $x,y\in\K$ such that $x^2+y^2=0$, then $(y/x)^2=-1$ and so
$\sqrt{-1}\in\K$. Thus for a field $\K$ the following assertions are equivalent:
\begin{enumerate}
 \item $\K$ has a square root of $-1$.
 \item There are nonzero elements $x,y\in\K$ such that $x^2+y^2=0$.
\end{enumerate}

\begin{de}
 We say that a field $\K$ is $2$-formally real if for any $x,y\in\K$, the equality $x^2+y^2=0$ implies $x=y=0$.
 More generally for an ideal $\mathfrak{i}$ of a ring $R$ we say that $R$ is $\mathfrak{i}$-$2$-formally real if for any $x,y\in R$, the
 fact $x^2+y^2\in\mathfrak{i}$ implies $x,y\in\mathfrak{i}$. 
\end{de}

For instance $\Q$, $\R$ and $\Z_{p}:=\Z/p\Z$ (with $p$ a prime of the form $p=4k+3$) are $2$-formally real while $\C$ and $\Z_p$ (with $p$ a prime of the form
$p=4k+1$) are not. We will devote a section to finite fields and there, we will test the $2$-formally real character of these fields. 
On the other hand, if $D$ is a product of $2$-formally real fields, then the ring $D$ is $\m$-$2$-formally real for any  maximal ideal
$\m\triangleleft D$. We will see that in this case the Lorentz type algebra $\L_D$ is $\m$-simple for every maximal ideal.
\medskip

Consider now the basis $\{b_i\colon i=1,\ldots, 6\}$ of $\L_R$ such that $b_1=s_{12}$, $b_2=s_{13}$, $b_3=s_{14}$, $b_4=a_{23}$, $b_5=a_{24}$ and
$b_6=a_{35}$. The multiplication table of $\L_R$ relative to this basis is the transcription of the one in Figure \ref{ttaudi}:
\begin{figure}[h]\label{audia4}
\[\begin{tabular}{|c|cccccc|}
\hline
$[\ ,\ ]$ & $b_1$ & $b_2$ & $b_3$ & $b_4$ & $b_5$ & $b_6$\\ 
\hline
$b_1$ & $0$ & $b_4$ & $b_5$ & $b_2$ & $b_3$ & $0$ \\
$b_2$& $-b_4$ & $0$ & $b_6$ & $-b_1$ & $0$ & $b_3$
   \\
$b_3$& $-b_5$ & $-b_6$ & $0$ & $0$ & $-b_1$ &
   $-b_2$ \\
$b_4$& $-b_2$ & $b_1$ & $0$ & $0$ & $-b_6$ & $b_5$
   \\
$b_5$& $-b_3$ & $0$ & $b_1$ & $b_6$ & $0$ & $-b_4$
   \\
$b_6$& $0$ & $-b_3$ & $b_2$ & $-b_5$ & $b_4$ & $0$\\
\hline
 \end{tabular}\]
 \caption{Second version of table in Figure \ref{ttaudi}.}
 \end{figure}\medskip

\begin{pr}\label{Notsimple}
 Let $R$ be a commutative unitary $\Phi$-algebra and $\m\in\max(R)$ such that $R$ is not $\m$-$2$-formally real, that is, 
 there are $x,y\in R$ satisfying $x^2+y^2\in\m$ but
 $x,y\not\in\m$. Define the elements 
 $$\begin{cases} a_1:=xb_1+yb_6,\cr a_2:=xb_2-yb_5,\cr a_3:=xb_3+yb_4\end{cases}$$
 of $\L_R$, then $I:=Ra_1+Ra_2+Ra_3+\m\L_R$ is a nontrivial nontotal ideal of $\L_R$ with an $\m$-free part of cardinal $3$. 
 Consequently $\L_R$ is not $\m$-simple in this case.
\end{pr}
\begin{proof}
 Since $x\not\in\m$ and $R/\m$ is a field, there is an $x'\in R$ such that $xx'\in 1+\m$.
 Next we compute the products $[a_1,b_i]$ using the symbol $\equiv$ for the relation of congruence module the ideal $\m\L_R$:
 \begin{itemize} \item $[a_1,b_1]=0$. 
 \item  $[a_1,b_2]=xb_4-yb_3\equiv x(b_4-x'yb_3)$ and since $y^2\equiv -x^2$, then $x'y^2\equiv -x'x^2\equiv -x$. Thus
 $[a_1,b_2]\equiv x(b_4-x'yb_3)\equiv xy'(yb_4-x'y^2b_3)\equiv xy'(yb_4+x'x^2b_3)\equiv xy'(yb_4+x b_3)$. So
 $[a_1,b_2]\in R a_3+\m\L_R$. Similarly:
 \item $[a_1,b_3]\in Ra_2+\m\L_R$, $[a_1,b_4]=a_2$, $[a_1,b_5]=a_3$ and $[a_1,b_6]=0$.
 \end{itemize}
   So far we have proved $[a_1,\L_R]\subset I$ and using the same type of computations we can also prove that
   $[a_i,\L_R]\subset I$ for $i=2,3$. Thus, $I$ is a nonzero ideal of $\L_R$. To prove that $I$ is nontotal consider
   the epimorphism $\phi\colon\L_R\to\L_R\otimes_R R/\m$. Let $\K:=R/\m$ be the corresponding field, then $\L_R\otimes_R R/\m=\L_\K$ and 
   $\phi(I)\subset \K (a_1\otimes 1)+\K (a_2\otimes 1)+\K (a_3\otimes 1)$ so that $\dim_\K\phi(I)\le 3$ hence $I$ is nontotal. 
Finally we prove that $\{a_1,a_2,a_3\}$ is a $\m$-free part: assume that $\sum r_ia_i\in\m$ (we must prove that each $r_i\in\m$).
Then $r_1(xb_1+yb_6)+r_2(xb_2-yb_5)+r_3(xb_3+yb_4)\in\m\L_R$ and so there are $m_i\in\m$ such that
$r_1(xb_1+yb_6)+r_2(xb_2-yb_5)+r_3(xb_3+yb_4)=\sum m_ib_i$. Thus $r_1x\in\m$ and $\m$ being a maximal ideal it is prime. 
Since $x\not\in\m$ then $r_1\in\m$. Similarly the remaining $r_j$'s are in $\m$.
\end{proof}

To finish this section we establish the dichotomy theorem

\begin{theo}\label{snss}
Let $\m\in\Max(R)$ be a maximal ideal. Then
the Lorentz type algebra $\L_R$ is $\m$-simple if and only if $R$ is $\m$-$2$-formally real.
In particular, the Lorentz type algebra $\L_\K$ over a field $\K$ is simple if and only if $\K$ is
$2$-formally real (equivalently, if and only if $\sqrt{-1}\in\K$).
\end{theo}
\begin{proof}
If $R$ is not $\m$-$2$-formally real, we have seen in Proposition \ref{Notsimple} that $\L_R$ is not $\m$-simple.
Next we prove that in case that $R$ is $\m$-$2$-formally real, then the Lorentz type algebra $\L_R$ is $\m$-simple.
For that, we start with the simpler case when $R$ is a field $\K$. Under the hypothesis in the theorem we know that
$\sqrt{-1}\not\in\K$.
Take the basis $B=\{b_i\}$ of $\L_\K$ defined in the introduction, and an arbitrary nonzero element
$g=\sum\lambda_i b_i$ where $\lambda_i\in\K$. Let $I=(g)$ be the ideal generated by $g$. It is easy to see that $I\ni z:=g-[[g,b_1],b_1]=x b\sb 1+y b\sb 6$ where $x=\lambda_1$ and $y=\lambda_6$.
\begin{itemize}
\item Assume that $x$ or $y$ is nonzero.
Then, if we compute the matrix whose rows are the coordinates of
$[z,b\sb i]$ for $i=1,\ldots,6$, we get (removing the null rows) the matrix
\[
 {\small \begin{pmatrix} 0 & 0 & -y & x& 0 & 0\cr
  0 & y & 0 & 0 & x & 0\cr
  0 & x & 0 & 0 & -y & 0\cr
  0 & 0 & x & y & 0 & 0
  \end{pmatrix}},\text{ and since } 
  {\small\left|
  \begin{matrix}   0 & -y & x& 0 \cr
    y & 0 & 0 & x \cr
   x & 0 & 0 & -y \cr
   0 & x & y & 0 \cr
  \end{matrix}\right| =(x^2+y^2)^2\ne 0}
\]
\noindent we get that the dimension of the image of $\ad(z)$ is $4$ and 
taking into account Theorem \ref{cuatrocinco}, we have that $I$ is the whole algebra.  
\item If $x=y=0$, then $g=\sum_{i=2}^5\lambda_i b_i$. One can see that 
defining now $z:=[g,b\sb 2]=x b\sb 1+y b\sb 6$ where $x=\lambda\sb 4$ and
$y=-\lambda\sb 3$, and repeating the argument above, we get that the ideal $I$ is the whole algebra or $x=y=0$. In this last case
$g=\lambda_2 b\sb 2+\lambda_5 b\sb 5$. But defining now $z:=[g,b\sb 3]=x b\sb 1+y b\sb 6$ for $x=\lambda\sb 5$ and $y=\lambda\sb 2$ we get
that again the ideal $I$ is the whole algebra or $x=y=0$ which is a contradiction
since $g\ne 0$.
\end{itemize}
This proves the simplicity of the Lorentz algebra if $R$ is a field.
In the general case, take an ideal $I\triangleleft \L_R$.
Consider as usual the homomorphism $\phi\colon\L_R\to\L_\K$ for $\K:=R/\m$ and $\phi(I)$, which is an ideal of $\L_\K$.
Since $\sqrt{-1}\not\in\K$ the $\K$-algebra $\L_\K$ is simple and so $\phi(I)=0$ or $\phi(I)=\L_\K$. Thus $I$ is either $\m$-null or $\m$-total.
This finishes the proof.
\end{proof}

\begin{re}\rm
Any field $\K$ of characteristic two has a square root of $-1=1$. Therefore, the Lorentz type algebras over fields of characteristic two are not simple.
\end{re}

To finish this section we would like to describe those semisimple algebras $R$ in $\alg_\Phi$ such that the Lorentz type algebra $\L_R$ is 
$\m$-simple for any $\m\in\Max(R)$. Here \lq\lq semisimple\rq\rq\  means $\rad(R)=0$ where $\rad(\cdot)$ denotes the Jacobson radical. 
Applying Theorem \ref{snss} the necessary and sufficient condition for this is that 
$R$ must be $\m$-$2$-formally real for any $\m\in\Max(R)$. Since there is a monomorphism $j:R\to\prod_{\ms\in\Max(R)} R/\m$ 
mapping any $x$ to $(x+\m)_{\ms\in\Max(R)}$, we have that $R$ is a subalgebra of a product of fields $\K_{\ms}:=R/\m$ which
are $2$-formally real. This is of course the usual description of $R$ as a subdirect product of fields $\prod_{\ms}\K_\ms$, but each
factor $\K_\ms$ is $2$-formally real. However not every subdirect product $R$ of $2$-formally real fields gives a Lorentz type algebra
$\L_R$ with the property that this algebra is $\m$-simple for every $\m$. We need a further property that we explain in the following result:

\begin{theo}
 Let $R$ be an algebra in $\alg_\Phi$ with Jacobson radical $\rad(R)=0$. Then $\L_R$ is $\m$-simple for any $\m\in\Max(R)$ if and only if 
 $R\subset\prod_{i\in I}\K_i$
 is a subdirect product of $2$-formally real field $\{\K_i\}_{i\in I}$ and for any $\m\in\Max(R)$ there is some
 $j\in I$ such that $\pi_j(\m)=0$, being $\pi_j\colon \prod_{i\in I}\K_i\to\K_j$ the canonical projection onto the field $\K_j$.
\end{theo}
\begin{proof}
 If $\L_R$ is $\m$-simple for every $\m$, we take $\{\K_i\}_{i\in I}=\{R/\m\}_{\ms\in\Max(R)}$ and the property is satisfied.
 Reciprocally assume $R\subset\prod_i \K_i$ is a subdirect product of $\{\K_i\}$ with the additional property on the maximal ideals $\m$.
 Let us prove that $R$ is $\m$-$2$-formally real for each $\m\in\Max(R)$. If we assume $x^2+y^2\in\m$ then, since there is some $j$ such
 that $\pi_j(\m)=0$, we have $\pi_j(x)^2+\pi_j(y)^2=0$ in the field $\K_j$ which is $2$-formally real. So $\pi_j(x),\pi_j(y)=0$ and
 $x,y\in \ker\pi_j=\m$.
\end{proof}

\begin{re}
\rm When $R$ turns out to be a field the above necessary and sufficient condition is that $R$ must be $2$-formally real.
\end{re}

\section{Decomposability of $\sl_2(R)\times\sl_2(R)$}

In this section we consider the possibility of decomposing the Lie algebra $\sl_2(R)\times\sl_2(R)$ as a direct sum of ideals which
are $\m$-simple algebras for any $\m\in\Max(R)$. Define $R^2:=R\times R$ with componentwise operations. We will identify
$\sl_2(R^2)$ with $\sl_2(R)\times\sl_2(R)$ in the standard way.

\begin{lm}\label{stinkone}
 Let $I$ be an ideal of the Lie ring $\sl_2(R^2)$, that is, $I+I\subset I$ and $[I,\sl_2(R^2)]\subset I$.
 Assume that $\frac{1}{2}\in R$.
 Then there are ideals $\im,\j\triangleleft R$ such that $I=\sl_2(\im)\times\sl_2(\j)$.
\end{lm} 
\begin{proof}
 Define $\im:=\{x\in \sl_2(R)\colon (x,0)\in I\}$ and $\j:=\{x\in \sl_2(R)\colon (0,x)\in I\}$. It is easy to see
 that $\im,\j\triangleleft R$ and also $\sl_2(\im)\times\sl_2(\j)\subset I$.
 Assume now that $(a,b)\in I$, then $[(a,b),\sl_2(R)\times 0]=[a,\sl_2(R)]\times 0\subset I$ hence $[a,\sl_2(R)]\subset\im$.
 From this, we can prove that $a\in\im$ in the following manner: write $a=r_1h+r_2e+r_3f$ where $\{h,e,f\}$ is the standard basis
 of $\sl_2(R)$ such that $[h,e]=2e$, $[h,f]=-2f$ and $[e,f]=h$. Then $I\ni [[e,[a,h]],f]=[[e,-2r_2e+2r_3f],f]=2r_3 [h,f]=-4r_3 f$.
 Thus $r_3f\in I$. With similar arguments we prove $r_2e,r_1h\in I$. Thus $a\in\im$. Dually, $b\in\j$ and so $I\subset\sl_2(\im)\times\sl_2(\j)$.
\end{proof}

\begin{lm}\label{stinktwo}
 Assume that $R$ is an algebra in $\alg_\Phi$ such that $\frac{1}{2}\in R$ and $R=\ia\oplus\ib$ where $0\ne\ia,\ib\triangleleft R$.
Consider the ideals $I:=\sl_2(\ia)\times\sl_2(\ib)$ and $J=\sl_2(\ib)\times\sl_2(\ia)$ of $\sl_2(R^2)$.
Then $\sl_2(R^2)=I\oplus J$ and:
\begin{enumerate}
\item $\ann_R(I)=\ann_R(J)=0$.
\item For each $\m\in\Max(R)$, the $R$-algebras  $I$ and $J$ are $\m$-simple.
\end{enumerate}
 \end{lm}
\begin{proof}
The assertion on the annihilators is straightforwad:  $\ann_R(I)=\ann(\ia)\cap\ann(\ib)=0$ because $\ia+\ib=R$.
In a dual form we have $\ann_R(J)=0$.
Let us prove that the $R$-algebra $I$ is $\m$-simple (for every $\m\in\Max(R)$). For this we consider
the canonical epimorphism $\phi\colon I\to I/\m I$ and we have to prove that for any ideal of $I$, its image under $\phi$
is either $0$ or the whole $I/\m I$. Since $\m$ is a maximal ideal we may not have $\ia\subset\m$ and $\ib\subset\m$.
 Assume that $\ia\not\subset\m$. Then $\ib\subset\m$, to see this, take into account that $R=\ia+\m$ hence $1=a+m$ for certain
 $a\in\ia$ and $m\in\m$. Then any $b\in\ib$ can be written as $b=ba+bm=bm\in\m$.
So $\m I=\sl_2(\m\ia)\times\sl_2(\m\ib)$ but $\m\ia=\m\cap\ia$ and
$\m\ib=\ib$. To prove this last assertion we use again that $1=a+m$ where $a\in\ia$ and $m\in\m$. 
Then if $x\in \ia\cap\m$ we have $x=xa+xm\in\m\ia$ proving that $\ia\cap\m\subset\ia\m$. On the other hand, to see that $\ib\subset\ib\m$
take $b\in\ib$. Then $b=ba+bm=bm\in\ib\m$. Consequently, $\m I=\sl_2(\ia\cap\m)\oplus\sl_2(\ib)$ and 
$I/\m I\cong\sl_2(\ia/\ia\cap\m)\times 0\cong \sl_2(R/\m)\times 0$ which is a simple $R/\m$-algebra (see Lemma \ref{lematres}). Thus
the image under $\phi$ of any ideal of $I$ must be zero or the whole algebra $I/\m I$.

 Dually any ideal of $J$ is $\m$-null or $\m$-total.
 \end{proof}

\begin{theo}\label{geo}
 Let $R^2:=R\times R$ with componentwise operations, $1/2\in R$ and assume that $\sl_2(R^2)$ splits  as
 $\sl_2(R^2)=I\oplus J$ where $I,J\triangleleft\sl_2(R^2)$ such that $\ann_R(I)=\ann_R(J)=0$. 
 Then there are ideals $\ia,\ib\triangleleft R$ such that $R=\ia\oplus\ib$ being $I=\sl_2(\ia)\times\sl_2(\ib)$
 and $J=\sl_2(\ib)\times\sl_2(\ia)$.
 
 \end{theo}
 
\begin{proof} Assuming that the decomposition is given,
 we know by Lemma \ref{stinkone} that $I=\sl_2(\ia)\times\sl_2(\ib)$ for some ideals $\ia,\ib\triangleleft R$. Similarly
 $J=\sl_2(\ic)\times\sl_2(\id)$ for some ideals $\ic,\id\triangleleft R$. But since $[I,J]=0$, then
 $[\sl_2(\ia),\sl_2(\ic)]=0$ and $[\sl_2(\ib),\sl_2(\id)]=0$ 
 which implies
$\ia\ic=0$ and $\ib\id=0$. From here it is easy to see that
$\hbox{ann}(\ia)=\ic$, $\hbox{ann}(\ic)=\ia$, $\hbox{ann}(\ib)=\id$ and $\hbox{ann}(\id)=\ib$.
On the other hand, $R=\ia\oplus\ic=\ib\oplus\id$. Also $0=\hbox{ann}(I)=\hbox{ann}(\ia)\cap\hbox{ann}(\ib)=\ic\cap\id$ and
similarly $0=\hbox{ann}(J)=\hbox{ann}(\ic)\cap\hbox{ann}(\id)=\ia\cap\ib$. Thus $\ic\id=0$ and $\ia\ib=0$ implying
$\ic\subset\hbox{ann}(\id)=\ib$ and $\ib\subset\hbox{ann}(\ia)=\ic$ hence $\ic=\ib$ and symmetrically $\ia=\id$.
So $I=\sl_2(\ia)\times\sl_2(\ib)$ and $J=\sl_2(\ib)\times\sl_2(\ia)$ being $R=\ia\oplus\ib$.
\end{proof}

\begin{re}\label{along}\rm
 Along the ideas in the proof of Theorem \ref{geo}, it is easy to prove the following result: If $R^2$ splits in form
 $R^2=I\oplus J$ where $I,J\triangleleft R^2$ such that $\hbox{ann}_R(I)=\hbox{ann}_R(J)=0$, then there are ideals $\ia,\ib\triangleleft R$ such that $R=\ia\oplus\ib$, 
 $I=\ia\times\ib$ and $J=\ib\times\ia$. The ideals $\ia$ and $\ib$ are necessarily unique.
\end{re}


\section{Lorentz type algebra over rings of characteristic other than $2$.}

In the first part of this section we consider an algebra $R$ in $\alg_\Phi$ such that $1/2\in R$ and also $\sqrt{-1}\in R$.
We have seen in Lemma \ref{one} that under these hypothesis,  the Lorentz type algebra over $R$ is isomorphic to $\o(4;R)$. 
We have proved also in Lemma \ref{lematres} that the $R$-algebra $\sl_2(R)$ is basically simple. We have also the following:

\begin{lm}\label{auxuno}
 Let $\frac{1}{2}\in R$ and $J\triangleleft \sl_2(R)$ an ideal. Assume $x\in\sl_2(R)$ satisfies $[x,\sl_2(R)]\subset J$.
 Then $x\in J$.
\end{lm}
\begin{proof}
 Let $x=r_1h+r_2e+r_3f$ where $\{h,e,f\}$ is the standard basis of the free $R$-algebra $\sl_2(R)$.
 Then $[x,h]=-2r_2e+2r_3f\in J$ and $J\ni [e,[x,h]]=2r_3h$ hence $r_3h\in J$ implying $r_3f\in J$.
 With similar arguments $r_2 e\in J$ and $r_1h\in J$. Consequently $x\in J$.
\end{proof}

The following result is standard over algebraically closed fields of characteristic other than $2$.

\begin{pr}\label{prtwo}
Let $R$ be an algebra with a square root of $-1$ and such that $1/2\in R$. Then for the
Lorentz type algebra $\L_R$ we have:
$$\L_R\cong \sl\nolimits_2(R)\oplus \sl\nolimits_2(R).$$
\end{pr}
\begin{proof}
Consider the Lie algebra  $\sl_2(R)$ with its usual basis 
$\{h:=e_{11}-e_{22}, e:=e_{12}, f:=e_{21}\}$, which satisfies $[h,f]=-2f$, $[h,e]=2e$ and  $[e,f]=h$.
Now, for $i,j\in\{1,2,3,4\}$ with $i\ne j$,  
define in $\o(4)$ the elements 
\begin{align}
   h_\alpha  = -\iu \left(a_{1,2}+a_{3,4}\right), & & h_\beta=\iu \left(a_{1,2}-a_{3,4}\right),\nonumber\\
   v_\alpha  =\frac{1}{2}
   (a_{1,3}-a_{2,4})+\frac{\iu}{2}
   (a_{1,4}+a_{2,3}), & & v_\beta=\frac{1}{2}
   \left(-a_{1,3}-a_{2,4}\right)+\frac{\iu}{2}
   \left(-a_{1,4}+a_{2,3}\right),\nonumber\cr
   v_{-\alpha} =\frac{1}{2}
   \left(-a_{1,3}+a_{2,4}\right)+
   \frac{\iu}{2}
   \left(a_{1,4}+a_{2,3}\right), & & v_{-\beta}=\frac{1}{2}
   \left(a_{1,3}+a_{2,4}\right)+\frac{\iu}{2}
   \left(-a_{1,4}+a_{2,3}\right),
\end{align}
where $a_{i,j}:=e_{ij}-e_{ji}$. These elements verify
\begin{equation}
[h_\a,v_\a]=2v_\a,\  
 [h_\a,v_{-\a}]=-2v_{-\a},\ 
 [v_\a,v_{-\a}]=h_\a,
\end{equation}
that is to say, the identities (\ref{idsl}). The isomorphism is now clear:
\[
\begin{matrix} \alpha:\sl_2(R)&\longrightarrow & I=<h_{\alpha},v_{\alpha},v_{-\alpha}> \cr
h &\longmapsto & h_{\alpha}\cr e&\longmapsto &v_{\alpha}\cr f&\longmapsto &v_{-\alpha} \end{matrix}
\]
We do the same to prove that $J=\langle h_{\beta},v_{\beta},v_{-\beta}\rangle$ is also isomorphic to $\sl_2(R)$.
Thus, any of the $3$-dimensional $R$-submodules $\langle h_\a,v_\a,v_{-\a}\rangle$
and $\langle h_\b,v_\b,v_{-\b}\rangle$ is a subalgebra isomorphic to $\sl_2(R)$ and they satisfy $[x,y]=0$ for any
$x\in \langle h_\a,v_\a,v_{-\a}\rangle$ and any $y\in \langle h_\b,v_\b,v_{-\b}\rangle$. Thus, we have 
$\o(4)=I\oplus J$ and $I,J\triangleleft\ \o(4)$ (ideals of $\o(4)$). Moreover, $I\cong\sl_2(R)\cong J$ and so we have
proved the proposition. 
\end{proof}

\begin{re}\rm
For a complex vector space $V$ we are denoting by $V^\R$ its \lq\lq reallyfication\rq\rq, that is the underlying real vector space 
of $V$. If $A$ is a complex algebra, by $A^\R$ we denote the underlying real algebra of $A$. An easy observation is that simplicity of
the complex algebra $A$ implies simplicity of the real algebra $A^\R$. In a more general fashion assume that $R$ is an algebra in $\alg_\Phi$.
If $S$ is an algebra in $\alg_\Phi$ and $R$ a subalgebra of $S$, then for any $S$-algebra $V$ we will denote by $V^R$ the restriction of scalars
algebra of $V$. 
\end{re}\medskip

For any algebra $R$ in $\alg_\Phi$, we may construct $\bar R:=R\times R$ where the product in $\bar R$ is given by 
\begin{equation}\label{dupli}(x,y)(u,v):=(xu-yv,xv+yu).\end{equation}
We have a canonical monomorphism $R\to\bar R$
such that $r\mapsto (r,0)$. Regardless of the fact that $\sqrt{-1}\in R$ or not, in the new algebra $\bar R$ we always have $\sqrt{-1}\in\bar R$. 
Furthermore, if $\frac{1}{2}\in R$ then $\frac{1}{2}\in \bar R$.

\begin{lm}
If $\frac{1}{2}\in R$ the automorphism group $\aut_R(\bar R)$ is isomorphic to $\mu_2(R)$ (the group of order-two elements in $R$).
The isomorphism $\theta_R\colon\aut_R(\bar R)\cong\mu_2(R)$ may be chosen to be natural in $R$: for any homomorphism of $\Phi$-algebras
$f\colon R\to S$ there is a commutative diagram 
\[
\xygraph{
!{<0cm,0cm>;<1cm,0cm>:<0cm,1cm>::}
!{(0,0)}*+{\aut_R(\bar R)}="a"
!{(2,0)}*+{\mu_2(R)}="b"
!{(0,1.6)}*+{\aut_S(\bar S)}="c"
!{(2,1.6)}*+{\mu_2(S)}="d"
"a":_{\theta_R}"b"
"c":^{\theta_S}"d"
"a":^{f^*}"c"
"b":_{f_*}"d"
}
\]
where $f_*$ is the restriction of $f$ from $\mu_2(R)$ to $\mu_2(S)$ and $f^*$ maps any $R$-autom\-orphism $\a$ of $\bar R$ to 
the $S$-automorphism $\b$ of $\bar S$ such
that $\b(0,1)=(f(x),f(y))$ where $\a(0,1)=(x,y)$.
\end{lm}
\begin{proof}
 Let $\a\in\aut_R(\bar R)$, then $\a(0,1)=(a,b)$ such that $a^2-b^2=-1$ and $ab=0$. Then $\alpha^{-1}(0,1)=(c,d)$ where $c^2-d^2=-1$ and $cd=0$. 
Since $\a^{-1}\a=1$ we have $(0,1)=\alpha^{-1}(a,b)=(a,0)+b(c,d)=(a+bc,bd)$. Thus $b,d\in R^\times$ and since $ab=0$, we have $a=0$ and $b^2=1$.
Thus any automorphism $\a$ acts in the form $\a(0,1)=(0,b)$ where $b^2=1$. So the map $\aut_R(\bar R)\to\mu_2(R)$ such that $\a\mapsto b$ is a
group isomorphism that we can denote by $\theta_R$. The assertion on the naturality of $\theta_R$ is straightforwad. 
 \end{proof}

 \begin{re}\label{langgrsch}\rm
  In the language of group schemes, the above lemma says that there is an isomorphism of affine group schemes 
  $\theta\colon\affaut(\bar\Phi)\cong \bm{\mu}_2$.
 \end{re}

\begin{lm}
If $\frac{1}{2}\in R$ the ideals of the $R$-algebra $\sl_2(\bar R)^R$ are exactly the ideals of the $\bar R$-algebra $\sl_2(\bar R)$.
\end{lm}
\begin{proof} We have to prove that if $I\triangleleft\sl_2(\bar R)^R$ then $\iu I\subset I$ where
$\iu=\sqrt{-1}\in \bar R$. Take $x\in I$, then $x=\alpha h+\beta e+\gamma f$ for some $\alpha,\beta,\gamma\in\bar R$. 
So $I\ni [x,\iu h]=-2\beta\iu e+2\gamma\iu f$ and also $I\ni [e,[x,\iu h]]=2\gamma\iu h$ hence 
$I\ni [f,[e,[x,\iu h]]]=4\gamma\iu f$ and so $\iu\gamma f\in I$.
\end{proof}

\begin{theo}\label{really}
Let $R$ be an algebra in $\alg_\Phi$ with $\frac{1}{2}\in R$. 
Take as before $\bar R$ and consider $\sl_2(\bar R)$. Then $$\L_R\cong\sl\nolimits_2(\bar R)^R.$$
In particular, the (real) Lorentz algebra $\L_\R$ is isomorphic to the reallyfication of $\sl_2(\C)$.
\end{theo}
\begin{proof} One of the square roots of $-1$ in $\bar R$ is $\iu=(0,1)$ and indeed $\bar R=R\oplus \iu R$ (as $R$-modules).
Then also $\sl_2(\bar R)=\sl_2(R)\oplus\iu \sl_2(R)$ as $R$-modules. Thus,
the system $\{h,e,f,\iu h,\iu e,\iu f\}$ is a basis of $\sl_2(\bar R)^R$.

We can define now a new basis in the Lie algebra $\L_R$, denoted by $ \mathcal{B}=\{x_i\}_{i=1}^6$, as follows:
\begin{eqnarray}
x_1:=-2(\alpha+\beta+w),\ x_2:=2(\alpha+v),\ x_3:=\alpha-u,\nonumber\\
x_4:=2(\alpha-u+v),\ x_5:=2(\beta+w),\ x_6:=\gamma+w,
\end{eqnarray}
where $\{\alpha,\beta,\gamma,u,v,w\}$ is the obvious basis obtained from the generic expression of an element in $\o(1,3)$: 
\begin{eqnarray}
\alpha=e_{12}+e_{21},\ \beta=e_{13}+e_{31},\ \gamma=e_{14}+e_{41},\nonumber \\
u=e_{23}-e_{32},\ v=e_{24}-e_{42},\ w=e_{34}-e_{43}.
\end{eqnarray}

We can verify now that denoting $h':=\iu h$, $e':=\iu e$, $f':=\iu f$, the multiplication table in $\sl_2(\bar R)^R$ is
\[
\begin{tabular}{|c|cccccc||}
\hline 
 & $h$ & $e$ & $f$ &  $h'$ & $e'$ & $f'$ \cr
\hline
$h$ & $0$ & $2e$ & $-2f$ & $0$ & $2e'$ & $-2f'$ \cr
$e$ & $-2e$ & 0 & $h$ &  $-2 e'$ & $0$ & $ h'$ \cr
$f$ & $2f$ & $-h $ & $0$ &  $2f'$ & $-h'$ & $0$ \cr
$h'$ & $0$ & $2e'$ & $-2f'$ & $0$ & $-2 e$ & $ 2 f$ \cr
$e'$ & $-2e'$ & $0$ & $h'$ & $ 2 e$ & $0$ & $-h $ \cr
$f'$ & $2f'$ & $-h'$ & $0$ & $-2f$ & $ h $ & $0$ \cr
\hline
\end{tabular}
\]
and the new one  in $\L_R$, with the new basis, is
\[
\begin{tabular}{|c|cccccc||}
\hline 
 & $x_1$ & $x_2$ & $x_3$ &  $x_4$ & $x_5$ & $x_6$ \cr
\hline
$x_1$ & $0$ & $2x_2$ & $-2x_3$ & $0$ & $2x_5$ & $-2x_6$ \cr
$x_2$ & $-2x_2$ & 0 & $x_1$ &  $-2 x_5$ & $0$ & $ x_4$ \cr
$x_3$ & $2x_3$ & $-x_1 $ & $0$ &  $2x_6$ & $-x_4$ & $0$ \cr
$x_4$ & $0$ & $2x_5$ & $-2x_6$ & $0$ & $-2 x_2$ & $ 2 x_3$ \cr
$x_5$ & $-2x_5$ & $0$ & $x_4$ & $ 2 x_2$ & $0$ & $-x_1 $ \cr
$x_6$ & $2x_6$ & $-x_4$ & $0$ & $-2x_3$ & $ x_1 $ & $0$ \cr
\hline
\end{tabular}
\]
It is clear then that  $\sl_2(\bar R)^R\cong \L_R$. 
\end{proof}

\section{Structure of Lorentz type algebras in characteristic two}
\label{secinco}

In order to study Lorentz type algebras in characteristic two we recall the functor $\o(3)\colon\alg_\Phi\to\lie_\Phi$
such that, for any algebra $R$ in $\alg_\Phi$, we define
$\o(3;R)$ as the free $R$-module  $\o(3;R)=R(e_{12}-e_{21})\oplus R(e_{13}-e_{31})\oplus R(e_{23}-e_{32})$ with the Lie algebra structure
induced by $[e_{12}-e_{21},e_{13}-e_{31}]=-(e_{23}-e_{32})$, $[e_{12}-e_{21},e_{23}-e_{32}]=e_{13}-e_{31}$ and
$[e_{13}-e_{31},e_{23}-e_{32}]=-(e_{12}-e_{21})$. We will call this, the $\o(3)$-type functor.
\begin{lm}
For a field $\K$ of characteristic $2$, the derivation algebra of $\o(3,\K)$
is isomorphic to the Lie algebra of symmetric $3\times 3$ matrices of zero trace with entries in $\K$.
Hence $\dim\der\o(3,\K)=5$. 
\end{lm}
\begin{proof}
The basis $b_1:=e_{12}+e_{21}$, $b_2=e_{13}+e_{31}$ and $b_3=e_{23}+e_{32}$ of $\o(3,\K)$ multiplies according to the rule
$[b_i,b_j]=b_k$ (where $i,j,k\in\{1,2,3\}$ cyclically). 
Taking $d\in\der\o(3,\K)$ and writing $d(b_i)=\sum x_{ij}b_j$, the equations
$d([b_i,b_j])=[d(b_i),b_j]+[b_i,d(b_j)]$ give: $$x_{ij}=x_{ji}, (i\ne j), \sum x_{ii}=0.$$
\end{proof}
\begin{pr}\label{hideput}
We consider now the Lorentz algebra ${\mathfrak L}:=\L_\K$ over 
 fields $\K$ of characteristic two. 
Then ${\mathfrak L}$ is not simple: it has a three dimensional ideal $I$ which is minimal and maximal, satisfies $[I,I]=0$
and $\mathfrak{L}/I\cong\o(3;\K)$ which is a simple Lie algebra. This ideal is unique.
\end{pr}

\begin{proof}
We consider the system $\{e_{ij}+e_{ji}\}\text{ such that } 1\leq i\neq j \leq 4$,   basis of $\mathfrak{L}$ and denote 
by $b_1:=s_{12}:=e_{12}+e_{21}$, $b_2:=s_{13}:=e_{13}+e_{31}$, $\ldots, b_6:=s_{34}:=e_{34}+e_{43}$. 
We take $I=\langle x_1,x_2,x_3\rangle$, where $x_1=b_1+b_6$, $x_2=b_2+b_5$ and $x_3=b_3+b_4$. 
It is easy to check that $I$ satisfies the following conditions:
\begin{itemize}
\item $[I,I]=0$. 
\item $[I,\mathfrak{L}]\subset I$
\end{itemize}
Then $I$ is an ideal of $\mathfrak{L}$.  To prove that $I$ is minimal we will prove that the ideal generated by any
element $g\in I$ is $I$ itself. This is trivial if $g=x_1$, $x_2$ or $x_3$.
So we take a generic element $0\ne g=\sum_i\lambda_i x_i\in I$. Denote the ideal generated by $g$ as $(g)$.
We have $[[g,b_1],b_3]=\lambda_3 x_1$.
 So if $\lambda_3\ne 0$ we have $x_1\in (g)$ and therefore
$I=(x_1)\subset(g)$ hence $(g)=I$. If on the contrary $\lambda_3=0$, then we consider the relation
$[[g,b_1],b_2]=\lambda_2 x_1$. Thus, if $\lambda_2\ne 0$ we have $x_1\in (g)$ and again $I=(x_1)\subset(g)$ implying
$(g)=I$. If $\lambda_2=0$, then $g=\lambda_1x_1\ne 0$, and so $(x_1)=(g)=I$. Summarizing: the ideal generated by any nonzero
element in $I$ is $I$ itself. So $I$ is minimal. Let us prove now that $I$ is also maximal. For this, we will
prove that the quotient algebra $\mathfrak{L}/I$ is simple. Let us denote by $\bar x$ the class of $x\in\mathfrak{L}$ module $I$.
Then a basis of the quotient algebra is $\{\bar b_1,\bar b_2,\bar b_3\}$ and the multiplication table of this algebra
is 
\[
 \begin{tabular}{|c|ccc||}
 \hline
 & $\stackrel{\ }{\bar b_1}$  & $\bar b_2$ & $\bar b_3$\cr
   \hline
   $\stackrel{\ }{\bar b_1}$  & $0$  & $\bar b_3$ & $\bar b_2$ \cr
   $\bar b_2$  &      & $0$  & $\bar b_1$ \cr
   $\bar b_3$  &      &  & $0$\cr
   \hline
 \end{tabular}
\]
On the other hand, the Lie algebra $\o(3;\K)$, as a vector space is the three-dimensional span
$\o(3;\K)=\langle e_{12}+e_{21},e_{13}+e_{31},e_{23}+e_{32}\rangle$ and if we make the multiplication table
of $\o(3;\K)$ relative to the specified basis we will see immediately the isomorphism $\mathfrak{L}/I\cong\o(3;\K)$.
Since $\o(3;\K)$ is simple we have also the maximality of the ideal $I$. To prove that $I$ is unique, assume that
$J$ is a different ideal satisfying the same properties as $I$. Then $J$ is also minimal and maximal, so that 
 $\L=I\oplus J$. Consequently 
$$[\L,\L]=[I+J,I+J]=[I,I]+[I,J]+[J,J]=0$$ which is a contradiction.
\end{proof}

\section{Lorentz algebra over finite fields}


In this section we investigate the simplicity of $\o(1,3)$ over a finite field $\F_q$.
Specifically, we consider the Lie algebra
$$L:=\L_{\F_q}=\left\{
{\begin{pmatrix} 
0 & x & y & z\cr
x & 0 & s & t\cr
y & -s & 0 & u\cr
z & -t & -u & 0
\end{pmatrix}}\colon x,y,z,s,t,u\in\F_q\right\}.$$
We know that $q$ must be a power $q=p^n$ of some prime number $p$ (and $n>0$).
Since we have already studied the Lorentz type algebra over fields of characteristic $2$ we assume
that $p$ is odd. By Theorem \ref{snss} the simplicity of $L$ is equivalent to the fact that the field $\F_q$ has no square root of $-1$.
\medskip

We start this subsection investigating under what conditions a finite field $\F_q$ (of characteristic other than $2$) has a square root
of $-1$. This is included here only for selfcontainedness reasons.

\begin{lm}\label{abuno}
 If $q$ is a positive integer,  the factorization 
 $x^q-x=(x^2+1)(x^{q-2}-x^{q-4}+\cdots+x^3-x)$ is only possible when $q$ is of the form $q=4n+1$.
\end{lm}
\begin{proof}
If $q=4n+1$ we can define the polynomial 
$$\sum_{k=1}^{2n} (-1)^{k+1} x^{q-2k}=x^{q-2}-x^{q-4}+\cdots+x^3-x$$
and the factorization holds (observe that we have used the fact that $q$ is of the form $4n+1$).
On the other hand, if $x^q-x$ is divisible by $x^2+1$, then the quotient is 
$x^{q-2}-x^{q-4}+\cdots +x^3-x$, so that the different summands are $(-1)^{k+1}x^{q-2k}$.
Equating $(-1)^{k+1}x^{q-2k}=-x$ we get that $k$ must be an even number $k=2n$ and $q-2k=1$. Hence $q=4n+1$. 
\end{proof}

\begin{pr}
Let $q=p^n$ where $p$ is an odd prime number, then $\F_q$ contains a square root of $-1$ if and only if $q$ is of the form $q=4n+1$.
\end{pr}
\begin{proof}
 Recall that $\F_q$ is the splitting field of $x^q-x$ over $\Z_p$. Thus all the elements in $\F_q$ satisfy $x^q-x=0$.
If $q=4n+1$ the factorization in Lemma \ref{abuno} holds and so there is a square root of $-1$ in the field.
Reciprocally, if $\sqrt{-1}\in\F_q$ then there is an element $q\in\F_q^\times$ of order $4$ (of course $q=\sqrt{-1}$).
Therefore the order of the group $\F_q^\times$ is a multiple of $4$. But this order is $q-1$ whence $q$ is of the form $4n+1$.
\end{proof}

\begin{co}
The Lorentz type algebra $\L_{\Z_p}$ over the field $\Z_p$ is simple if and only if $p$ is odd and of the form $p=4k+3$.
\end{co}

Finally we investigate when is an odd prime power $p^n$ of the form $4n+1$.

\begin{pr}
 Let $p$ be an odd prime number. Then:
 \begin{itemize}
  \item If $p=4k+1$, then $p^n$ is also of the form $4m+1$.
  \item If $p=4k+3$, then $p^n$ is of the form $4m+1$ if and only if $n$ is even.
 \end{itemize}
\end{pr}

\begin{proof}
Define $A$ to be the positive integers of the form $4n+1$ and $B$ those of the form $4n+3$.
Since $A$ is closed under multiplication the first assertion is trivial.
On the other hand $B B\subset A$ and $A B\subset B$ hence multiplying an even number of elements
of $B$ we get an element of $A$: $$\overbrace{B \cdots B}^{2k} \subset A,\quad\text{ and }\quad \overbrace{B \cdots B}^{2k+1} \subset B.
$$
\end{proof}

\begin{co}
 Consider the Lorentz type algebra $\L_{\F_q}$ over a finite field $\F_q$ where $q=p^n$ and $p$ is odd. Then $\L_{\F_q}$ is simple 
 if and only if $n$ is odd and $p$ of the form $p=4k+3$.
\end{co}


\section{Automorphisms and derivations of the Lorentz type algebra $\L_R$ if $\frac{1}{2}\in R$} 

Consider a ring $R$ with $\frac{1}{2}\in R$, then the algebra $\sl_2(R)$ admits a basis (as a free $R$-module) given by
$B=\{h,e,f\}$ where 
$h=\tiny\begin{pmatrix}1 & 0\cr 0 & -1\end{pmatrix}=e_{11}-e_{22}$, $e=\tiny\begin{pmatrix}0 & 1\cr 0 & 0\end{pmatrix}=e_{12}$ and
$f=\tiny\begin{pmatrix}0 & 0\cr 1 & 0\end{pmatrix}=e_{12}$. It is easy to check that if $d\colon\sl_2(R)\to\sl_2(R)$
is a derivation of the $R$-algebra $\sl_2(R)$, then its matrix relative to the basis $B$ is $$\left(
\begin{array}{ccc}
 0 & -2 \mu  & -2 \gamma  \\
 \gamma  & \lambda  & 0 \\
 \mu  & 0 & -\lambda 
\end{array}
\right)$$
so that the derivation algebra $\der(\sl_2(R))$ is the subalgebra of $\sl_3(R)$ generated as a free $R$-module by the matrices
$e_{22}-e_{33}$, $-2e_{12}+e_{31}$ and $-2e_{13}+e_{21}$. A routinary computation reveals that this subalgebra is isomorphic to $\sl_2(R)$
and so we have $$\der(\sl\nolimits_2(R))\cong\sl\nolimits_2(R).$$ Also this implies that all the derivations are inner in $\sl_2(R)$.

Consider now an algebraically closed field $K$ of characteristic other than $2$. 
It is easy to check that an nonzero element $M\in\sl_2(K)$ is semisimple if and only if $\vert M\vert\ne 0$. Also, it can be proved that the disjoint union of the 
orbits of elements $\tiny k\begin{pmatrix}1 & 0 \cr 0 & -1\end{pmatrix}$, ($k\in K^\times$) under the action of $\PGL_2(K):=\GL_2(K)/K^\times$ by conjugation is the set of all 
nonzero semisimple elements
or equivalently the set of all invertible matrices of $\sl_2(K)$. Recall that the group of inner automorphisms of $\sl_2(K)$ (denoted
$\inn(\sl_2(K))$) is just the image of  
the adjoint action $\Ad\colon \PGL_2(K)\to\aut(\sl_2(K))$.\medskip

The next two results are included only for the reader's convenience. Though they are known we have not found a suitable reference. 
They are extended in Lema \ref{noname} to the setting of algebraic groups in (affine scheme version). 

\begin{lm}\label{dfgdfg} Under the hypothesis in the previous paragraph,
take $h=\tiny\begin{pmatrix}1 & 0 \cr 0 & -1\end{pmatrix}$. If $\theta\in\aut(\sl_2(K))$ is such that $\theta(h)\in Kh$, then $\theta$ is inner.
\end{lm}
\begin{proof} Since $\theta(h)=\alpha h$ for some nonzero scalar $\alpha$, then $\theta(e)$ is an eigenvector of $\ad(h)$ with eigenvalue $2\alpha^{-1}$.
Since the nonzero eigenvalues of $\ad(h)$ are $\pm 2$, we have $\alpha=\pm 1$. We analyze both possibilities.
\begin{itemize}
 \item{$\alpha=1$}. We know that $\theta(e)\in Ke $ and similarly $\theta(f)\in Kf$. The matrix of $\theta$ relative to the basis $\{h,e,f\}$ is
 of the form
 $\hbox{diag}(1,\lambda,\lambda^{-1})$, for some nonzero scalar $\lambda$. This implies $\theta=\Ad\tiny\begin{pmatrix}\lambda & 0\cr 0 & 1\end{pmatrix}$, hence
 $\theta\in\inn(\sl_2(K))$.
 \item{$\alpha=-1$.} In this case $\theta(h)=-h$ and there is a nonzero scalar $\lambda$ such that $\theta(e)=\lambda f$ and $\theta(f)=\lambda^{-1}e$.
 Now it is straightforwad to prove that $\theta=\Ad{\tiny \begin{pmatrix}0 & \lambda^{-1}\cr 1 & 0\end{pmatrix}}$, hence $\theta\in\inn(\sl_2(K))$ again.
\end{itemize}
\end{proof}

\begin{co}\label{zuru}
 For an algebraically closed field of characteristic other than $2$ we have $\aut(\sl_2(K))=\inn(\sl_2(K))$.
 Moreover, the adjoint action $\PGL_2(K)\to\aut(\sl_2(K))$ is an isomorphism.
\end{co}
\begin{proof} Let $\psi\in\aut(\sl_2(K))$ and consider $h':=\psi(h)$ which is a semisimple element. Then some nonzero scalar multiple of $h$ 
is in the orbit of $h'$ under the (adjoint) action of 
$\PGL_2(K)$ hence for some inner automorphism $\Gamma$ we have $\Gamma(h')=\lambda h$ (for a suitable $\lambda\in K^\times$). Thus $\theta:=\Gamma\psi$ 
is in the hypothesis of Lemma \ref{dfgdfg}.
So $\theta$ in inner implying that $\psi$ is also inner. This amounts to saying that the adjoint action is an epimorphism.
But it is also a monomorphism since any $2\times 2$ matrix commuting with all the elements of $\sl_2(K)$ is necessarily a scalar multiple of the identity matrix.
The corollary is proved.\end{proof}\medskip

Consider now the category of groups $\grp$ and the affine group scheme $$\affpgl\nolimits_2\colon\alg\nolimits_\Phi\to\grp$$
such that for any algebra $R$ in $\alg_\Phi$ we have $\affpgl_2(R):=\affgl_2(R)/\affgm(R)$, where $\affgm(R):=R^\times$ is the affine
group scheme whose associated Hopf algebra is $\Phi[x,x^{-1}]$ (Laurent polynomials in $x$) and $\affgl_2$ the general linear affine group scheme.
We consider also the affine group scheme $\affaut(\sl_2(\Phi))\colon\alg_\Phi\to\grp$ such that $\affaut(\sl_2(\Phi))(R):=\aut_R(\sl_2(R))$, 
and the adjoint action $\affAd\colon \affpgl_2\to\affaut(\sl_2(\Phi))$ such that for any $R$ in $\alg_\Phi$ we have 
$\affAd_R\colon \PGL_2(R)\to\aut_R(\sl_2(R))$ in the usual sense.

\begin{lm}\label{noname} If $\Phi$ is a field of characteristic other than $2$, the adjoint map $\affAd$ is an isomorphism of algebraic group schemes
$\affpgl_2\to\affaut(\sl_2(\Phi))$.	
\end{lm}
\begin{proof} 
We know that $\dim\affaut(\sl_2(\Phi))=3$ hence this is a smooth affine group scheme.
We know that the differential $\ad\colon \pgl_2(\Phi)\to \der(\sl_2(\Phi))$ of $\affAd$ is an isomorphism (under
the hypothesis in the Lemma $\pgl_2(\Phi)\cong\sl_2(\Phi)$).
If $K$ is the algebraic closure of $\Phi$, applying Corollary \ref{zuru} and 
\cite[Proposition (22.5), p. 340]{boi} we get the required isomorphism.\end{proof}\medskip

The affine group scheme $\mu_n\colon\alg_\Phi\to\grp$ is defined as the one such that for any algebra $R$ in $\alg_\Phi$, we have
$\mu_n(R):=\{x\in R^\times\colon x^ n=1\}$. It is an algebraic group whose representing Hopf algebra is $\Phi[x]/(x^n-1)$.

\begin{lm}\label{carras}
In the case of a field $\K$ with $1/2,\sqrt{-1}\in\K$,
the derivation algebra of the Lorentz type algeba $\L_\K$ is isomorphic to $\sl_2(\K)^2$.
The automorphism group $\aut(\L_\K)$ is $\PGL_2(\K)^2\rtimes\mu_2(\K)$.
\end{lm}
\begin{proof} By Proposition \ref{prtwo}, we have $\L_\K\cong \o(4;\K)=I\oplus J$ for two ideals $I,J\cong\sl_2(\K)$. Since $[I,I]=I$ (and also $[J,J]=J$), these
ideals are invariant under derivations. Thus, $\der(\L_\K)=\der(I)\oplus\der(J)\cong\der(\sl_2(\K))^2=\sl_2(\K)^2$ taking into account
that the derivations of $\sl_2(\K)$ are inner. Now there are two types of automorphisms of $\L_\K$: those that fix the ideals $I$ and $J$ and those
that swap them. Consider first the (normal) subgroup $G_0$ of automorphism fixing the ideals. This is of course  $G_0=\aut(I)\times\aut(J)\cong
\aut(\sl_2(\K))^2=\PGL_2(\K)^2$ (which is a connected group). The automorphisms swapping the ideals $I$ and $J$ are those in 
$G_0\omega$ where
$\omega\in\aut(\L_\K)$ is the exchange automorphism acting in the way $\omega(h_\alpha)=h_\beta$, $\omega(v_{\pm\alpha})=v_{\pm\beta}$,
$\omega(h_\beta)=h_\alpha$, $\omega(v_{\pm\beta})=v_{\pm\alpha}$. Thus $\aut(\L_\K)=G_0\cup G_0\omega\cong G_0\rtimes\mu_2(\K)\cong
\PGL_2(\K)\rtimes\mu_2(\K)$.
\end{proof}

\begin{co}\label{inner}\label{interior}
 The derivations of $\L_\K$ are inner. The automorphisms of $\L_\K$ are inner, more precisely if we identify $\L_\K$ with $\sl_2(\K)^2$, 
 any $\phi\in\aut(\L_\K)$ fixing $I$ and $J$ is of the form $$\begin{pmatrix}x & 0\cr 0 & y\end{pmatrix}\mapsto 
 \begin{pmatrix}pxp^{-1} & 0\cr 0 & qyq^{-1}\end{pmatrix}=
 \begin{pmatrix}p & 0\cr 0 & q\end{pmatrix}\begin{pmatrix}x & 0\cr 0 & y\end{pmatrix}\begin{pmatrix}p^{-1} & 0\cr 0 & q^{-1}\end{pmatrix},$$
 and any $\phi\in\aut(\L_\K)$ swapping $I$ and $J$ is of the form
 $$\begin{pmatrix}x & 0\cr 0 & y\end{pmatrix}\mapsto 
 \begin{pmatrix}pyp^{-1} & 0\cr 0 & qxq^{-1}\end{pmatrix}=
 \begin{pmatrix}0 & p\cr q & 0\end{pmatrix}\begin{pmatrix}x & 0\cr 0 & y\end{pmatrix}\begin{pmatrix}0 & q^{-1}\cr p^{-1} & 0\end{pmatrix},$$
 for some $p,q\in\GL_2(\K)$.
\end{co}
\def\Lie{\mathop\text{\bf Lie}}

In this paragraph we use the finite constant group (see \cite[2.3, p. 16]{Waterhouse}) $\bm{Z}_2$. This is the affine group scheme whose
representing Hopf algebra is $\Phi^2:=\Phi\times\Phi$ with componentwise product and Hopf algebra structure define in the reference above.
Thus, for $R$ in $\alg_\Phi$ we have $\bm{Z}_2(R):=\hom_\Phi(\Phi^2,R)$ and so any element $f\in\bm{Z}_2(R)$ is completely determined by $e_1:=f(1,0)$
and $e_2:=f(0,1)$ which are orthogonal idempotents in $R$. This produces a decomposition $R=Re_1\oplus Re_2$ as a direct sum of two ideals
and reciprocally any set $\{e_1,e_2\}$ of orthogonal idempotents of $R$ gives a decomposition $R=Re_1\oplus Re_2$, hence a homomorphism
$f\colon\Phi^2\to R$ such that $f(1,0)=e_1$ and $f(0,1)=e_2$. Of course, if $R$ has no idempotents others than $0$ and $1$, the abstract
group of $\bm{Z}_2(R)$ is isomorphic to the group $\Z_2$ of integers module $2$. 
The set $\bm{Z}_2(R)$ is in one-to-one correspondence with the set of decompositions of $R$ as a direct sum of ideals.

\begin{re}\label{standard_result}\rm
It is a standard result that if $\frac{1}{2}\in\Phi$, there is an isomorphism of affine group schemes $\bm{Z}_2\cong\bm{\mu}_2$.
\end{re}

Consider next the affine group scheme $\affaut(\Phi^2)\colon\alg_\Phi\to\grp$ such that $$\affaut(\Phi^2)(R):=\aut\nolimits_R(R^2)$$ for any
 algebra $R$ in $\alg_\Phi$ (the $R$-algebra structure of $R^2$ is given by the componentwise product).
 There are some remarkable elements in $\aut_R(R^2)$ which deserve some attention. 
 If $\theta\in\aut_R(R^2)$ and $\theta(R\times 0)=R\times 0$, $\theta(0\times R)=0\times R$, then $\theta=1_{R^2}$.
 If $\theta$ swaps the ideals $R\times 0$ and $0\times R$ of $R^2$, then $\theta$ is the exchange automorphism.
 But in general $I:=\theta(R\times 0)$ and $J=\theta(0\times R)$ give a decomposition $R^2=I\oplus J$ where $I,J\triangleleft R^2$.
 Applying the result in Remark \ref{along} we have $I=\ia\times\ib$, $J=\ib\times\ia$ for some $\ia,\ib\triangleleft R$ such that
 $R=\ia\oplus\ib$. Therefore $\theta(r,0)=(\alpha_1(r),\alpha_2(r))$ where $\alpha_1\colon R\to\ia$, $\alpha_2\colon R\to \ib$.
 Also $\theta(0,r)=(\beta_1(r),\beta_2(r))$ where $\beta_1\colon R\to\ib$, $\beta_2\colon R\to \ia$.

 Reciprocally, for every decomposition
 $R=\ia\oplus\ib$ where $\ia,\ib\triangleleft R$,  we may consider the element $\sigma_{\ia,\ib}\in\aut_R(R^2)$ such that
 $\sigma_{\ia,\ib}(a+b,a'+b'):=(a+b',a'+b)$. This is an order-two automorphism and $\sigma_{\ia,\ib}\theta\in\aut_R(R^2)$
 preserves the ideals $R\times 0$ and $0\times R$, hence $\sigma_{\ia,\ib}\theta=1$ implying $\theta=\sigma_{\ia,\ib}$.
 Thus, there is a one-to-one correspondence from $\aut_R(R^2)$ to the set of all possible decompositions $R=\ia\oplus\ib$
 with $\ia,\ib\triangleleft R$.  Summarizing: we have an isomorphism 
\begin{equation}\label{finicons}
\affaut(\Phi^2)\cong\bm{Z}_2. 
\end{equation}

Next we want to investigate the group $\aut_R(\sl_2(R^2))$ for any algebra $R$ in $\alg_\Phi$. The first
obvious fact is that there is a group monomorphism $\aut_R(R^2)\to\aut_R(\sl_2(R^2))$ mapping $\alpha\mapsto \hat\alpha$ where $\hat\a$ consists on
applying $\alpha$ to all the entries of any matrix in $\sl_2(R^2)$. We will also use the notation $p^\a$ for the matrix 
$p^\a=(\a(p_{ij}))$ if $p$ is any matrix $p=(p_{ij})\in\gl_2(R^2)$.

Another source of automorphisms of $\sl_2(R^2)$ is given by the set
of all $f\times g$ where $f,g\in\aut(\sl_2(R))$ and under the usual identification of $\sl_2(R^2)$ with $\sl_2(R)^2$ we have
$(f\times g)(x,y):=(f(x),g(y))$ for any $x,y\in\sl_2(R)$. If $\Phi$ is a field of characteristic other than $2$, then $f=\Ad(\bar p)$
and $g=\Ad(\bar q)$ for $\bar p, \bar q\in\PGL_2(R)$. Thus $f\times g=\Ad(\bar p,\bar q)$ and we have the formula
\begin{equation}
\hat\a \Ad(\bar p,\bar q)=\Ad[(\overline{p},\overline{q})^\a] \hat\a  
\end{equation}
Thus, if $S_1$ is the subgroup of $\aut_R(\sl_2(R))$ given by
$S_1:=\{\hat\a\colon \a\in\aut_R(R^2)\}\cong\aut_R(R^2)$ and $S_2:=\{\Ad(\bar p,\bar q)\colon \bar p,\bar q\in\PGL_2(R)\}\cong\PGL_2(R)^2$,
we have $S_1,S_2\subset\aut_R(\sl_2(R^2))$ and $S_1\cap S_2=\{1\}$ by Lemma \ref{plafter}. 

\begin{theo}\label{complex_case}
 Let $\Phi$ be a field with $\frac{1}{2},\sqrt{-1}\in\Phi$ and $\L_\Phi$ the Lorentz type algebra over $\Phi$. Then the
 algebraic group $\affaut(\L_\Phi)$ is isomorphic to $\affpgl_2^2\rtimes\bm{Z}_2$. 
\end{theo}
\begin{proof}
Under the hypothesis above, for any $R$ in $\alg_\Phi$, the Lorentz type algebra $\L_R$ is isomorphic to $\o_4(R)\cong\sl_2(R)^2$ by  
Lemma \ref{one} and Proposition \ref{prtwo}.
Now, any $\alpha\in\aut_R(R^2)$ induces $\hat\alpha\colon \sl_2(R^2)\to\sl_2(R^2)$ applying $\alpha$ to all the entries of 
any matrix in $\sl_2(R^2)$. After identifying $\sl_2(R^2)$ with $\sl_2(R)\oplus\sl_2(R)$ we have that $\hat\alpha$ is 
an automorphism of $\sl_2(R)\oplus\sl_2(R)$. Next, to any couple $(\bar p,\bar q)\in\PGL_2(R)^2$ we may associate an
automorphism $\Ad(\bar p,\bar q)\colon \sl_2(R)\oplus\sl_2(R)\to \sl_2(R)\oplus\sl_2(R)$ such that $(x,y)\mapsto (\Ad(p)x,\Ad(q)y)=(pxp^{-1},qyq^{-1})$.
Thus, we can define  $F_R\colon \PGL_2(R)^2\rtimes\aut_R(R^2)\to\aut_R(\o_4(R))$ such that for any couple of equivalence classes
$(\bar p,\bar q)\in\PGL_2(R)$ and for any $\alpha\in\aut_R(R^2)$ we have $F_R((\bar p,\bar q),\alpha):=\Ad(\bar p,\bar q) \hat\alpha$.
So, we have defined a homomorphism of group schemes $F\colon \affpgl_2^2\rtimes\affaut(\Phi^2)\to \affaut(\L_\Phi)$ and next we show
that it is an isomorphism. For this, we construct an inverse $G\colon \affaut(\L_\Phi)\to \affpgl_2^2\rtimes\affaut(\Phi^2)$.
We take $\theta\in \aut_R(\L_R)$ which after the above identification is an automorphism $\theta\colon \sl_2(R^2)\to\sl_2(R^2)$.
Define $I:=\theta(\sl_2(R)\times 0)$ and $J=\theta(0\times\sl_2(R))$ which are ideals in $\sl_2(R^2)$ satisfying 
$\ann_R(I)=\ann_R(J)=0$ and $\sl_2(R^2)=I\oplus J$. Then, applying Theorem \ref{geo}, there are ideals $\ia,\ib\triangleleft R$
such that $R=\ia\oplus\ib$, $I=\sl_2(\ia)\times\sl_2(\ib)$ and $J=\sl_2(\ib)\times\sl_2(\ia)$. Next we define the
order-two automorphism $\tau_{\ia,\ib}$ of $\sl_2(R^2)$ such that $(a+b,a'+b')\mapsto (a+b',a'+b)$, where
$a,a'\in\sl_2(\ia)$, $b,b'\in\sl_2(\ib)$. Now, the composition $\tau_{\ia,\ib}\theta$ is an automorphism of $\sl_2(R^2)$ which 
preserves the ideals $\sl_2(R)\times 0$ and $0\times\sl_2(R)$. Therefore, there are $R$-automorphisms $f, g\in\sl_2(R)$
such that $\tau_{\ia,\ib}\theta=f\times g\colon \sl_2(R^2)\to\sl_2(R^2)$ such that $(x,y)\mapsto(f(x),g(y))$.
Thus $\theta=\tau_{\ia,\ib}(f\times g)$. Consider now $\sigma\colon R^2\to R^2$ such that 
$\sigma(a+b,a'+b')=(a+b',a'+b)$ for $a,a'\in\ia$, $b,b'\in\ib$. Then $\tau_{\ia,\ib}=\hat\sigma$ where 
$$\hat\sigma \tiny\begin{pmatrix}z_1 & z_2\\ z_3 & -z_1\end{pmatrix}=\begin{pmatrix}\sigma(z_1) & \sigma(z_2)\\ \sigma(z_3) & -\sigma(z_1)
\end{pmatrix}, z_i\in R^2, (i=1,2,3).$$
Furthermore, by Lemma \ref{noname}, $f\times g=\Ad(\bar p_1,\bar p_2)$ for some $\bar p_1,\bar p_2\in\PGL_2(R)$.
Thus $\theta=\hat\sigma\Ad(\bar p_1,\bar p_2)=\Ad(\bar p,\bar q)\hat\sigma$ where $(p,q)=(p_1,p_2)^\alpha$. Then
$$G_R(\theta):=((\bar p,\bar q),\sigma)$$
and so $F_R G_R=1$. The fact that $G_R F_R=1$ follows from the unicity of the expresion of an element in $\affaut_R(\L_R)$ as a composition
of an element of $S_2$ with an element in $S_1$ (Lemma \ref{plafter}).\end{proof}
\medskip

To finish this section we study the algebraic group $\affaut(\L_\Phi)$ in the case in which $\frac{1}{2}\in\Phi$ but $\sqrt{-1}\notin\Phi$.
\medskip

\begin{theo}\label{real_case}
Let $\Phi$ be a field with $\frac{1}{2}\in\Phi$ and with no square root of $-1$. Then the affine group scheme $\affaut(\L_\phi)$ is isomorphic
to $\affpgl_2(\bar\Phi)\rtimes\bm{Z}_2$. In particular, the automorphism group of the (real) Lorentz algebra is 
$\PGL_2(\C)\rtimes\mu_2(\R)$.
\end{theo}
\begin{proof}
 Let $R$ be in $\alg_\Phi$ and $\theta\in\aut_R(\L_R)$. We know that $\L_R=\sl_2(\bar R)^R$ by Theorem \ref{really}.  
We consider the usual basis $\{h,e,f\}$ of the free $\bar R$-module $\sl_2(\bar R)$. Then it can be proved that $h':=\theta(h)$, $e':=\theta(e)$ and $f':=\theta(f)$
yield a basis $\{h',e',f'\}$ of the $\bar R$-module $\sl_2(\bar R)$. So the $\bar R$-linear map $\omega\colon\sl_2(\bar R)\to\sl_2(\bar R)$ induced 
by $h'\mapsto h$, $e'\mapsto e$ and $f'\mapsto f$ is an $\bar R$-linear automorphism and $\alpha:=\omega\theta$ is an $R$-linear automorphism of 
$\sl_2(\bar R)$ fixing $h$, $e$ and $f$. It is easy now to see that there is an automorphism $\sigma\in\aut_R(\bar R)$ such that
$\alpha(\lambda h)=\sigma(\lambda)h$ for any $\lambda\in\bar R$. From this, it can be proved that also $\alpha(\lambda e)=\sigma(\lambda)e$ and
$\alpha(\lambda f)=\sigma(\lambda)f$ for any $\lambda\in\bar R$. We have proved that $\theta$ is the composition of a $\bar R$-linear automorphism of
$\sl_2(\bar R)$ with an automorphism induced by an element in $\aut_R(\bar R)$. In other words 
$\aut_R(\L_R)\cong \aut_{\bar R}(\sl_2(\bar R))\rtimes\aut_R(\bar R)\cong\PGL_2(\bar R)\rtimes\aut_R(\bar R)$
and the affine group scheme $\affaut(\L_\Phi)$ is isomorphic to $\affpgl_2(\bar \Phi)\rtimes\affaut_\Phi(\bar\Phi)$ and taking into account
remarks \ref{standard_result}
and \ref{langgrsch} we have $\affaut(\L_\phi)\cong
\affpgl_2(\bar\Phi)\rtimes\bm{Z}_2$.
\end{proof}

\begin{co}
If $\Phi$ is a field of characteristic other than $2$, the affine group scheme $\affaut(\L_\Phi)$ is smooth.
\end{co}
\begin{proof} If $\sqrt{-1}\in\Phi$ we apply Theorem \ref{complex_case} and then $\affaut(\L_\Phi)\cong\affpgl_2^2\rtimes \bm{Z}_2$ hence $\dim\affaut(\L_\Phi)=
\dim \affpgl_2^2 = 6$ since $\dim\bm{Z}_2=0$. If $\sqrt{-1}\not\in\Phi$ we apply Theorem \ref{real_case} and so 
$\affaut(\L_\Phi)\cong
\affpgl_2(\bar\Phi)\rtimes\bm{Z}_2$. Thus $\dim\affaut(\L_\Phi)=\dim \affpgl_2(\bar\Phi)$ (recall definition \ref{dupli}). Next we compute $\dim\affgl_2(\bar\Phi)$ where
$\affgl_2(\bar\Phi)(R):=\GL_2(\bar R)$. In fact $\affgl_2(\bar\Phi)$ is the affine group scheme represented by the Hopf algebra
$$H=\Phi[x_{11},x_{12},x_{21},x_{22},y_{11},y_{12},y_{21},y_{22},z_1,z_2]/I$$ where $I$ is the ideal generated by the polynomials
$pz_1-qz_2-1$ and $pz_2+qz_1$ being $$p=x_{11}x_{22}-x_{12}x_{21}-y_{11}y_{22}+y_{12}y_{21},
q=x_{11}y_{22}+x_{22}y_{11}-x_{12}y_{21}-x_{21}y_{12}.$$
So, applying again \cite[16.1 (a)]{Milne}, we have $\dim \affgl_2(\bar\Phi)=8$ and the kernel of the canonical epimorphism
$\affgl_2(\bar\Phi)\to \affpgl_2(\bar\Phi)$ is the affine group scheme $\alg_\Phi\to \grp$ such that $ R\mapsto (\bar R)^\times$.
The representing Hopf algebra of this functor is $\Phi[x_1,x_2,x_3,x_4]/J$ where $J$ is the ideal generated by the polynomials
$x_1x_3-x_2x_4-1$, $x_1x_4+x_2x_3$. Thus this kernel has dimension $2$ proving that $\dim \PGL_2(\bar\Phi)=6$.
\end{proof} 

To finish this section we observe that for an algebraically closed field $\Phi$ of characteristic other than $2$, one has  
\begin{equation}\label{bmw}
\dim\aut(\L_\Phi)=6.
\end{equation}

\section{Automorphisms of $\L_\K$ in the case of characteristic 2}


In this section the ground field $\K$ is algebraically closed of characteristic $2$ and
the Lie $\K$-algebra $\o(3;\K)$ (as introduced at the beginning of section \ref{secinco}), will be denoted simply by $\o(3)$ if
the reference to the field is not crucial . 
Also the notation $\O(3)$ will stand for the algebraic group  of $3\times 3$ matrices $M\in\GL_3(\K)$ such that $MM^t=1$. 
\begin{pr}
The algebraic group $\O(3)$ is connected and of dimension $3$.
\end{pr}
\begin{proof}
First of all, note that $\hbox{O}(3)$ is a connected algebraic group. In fact, $M\in\hbox{O}(3)$ if and only if 
$$M=\begin{pmatrix} a & b & 1+a+b\cr
    s & t & 1+s+t\cr
 1+a+s & 1+b+t  & 1+ a +b+ s+t    
   \end{pmatrix},$$
where $1+a+b+s+t+a t+b s=0$. Thus, as an affine variety the representing Hopf algebra of the affine group scheme 
whose group of $\K$-points is $\hbox{O}(3)$ is 
$H:=\K[a,b,s,t]/I$ where
$I$ is the ideal generated by $1+a+b+s+t+a t+b s$. But since the ideal $I$ is prime $H$ has no idempotents other that $0$ and $1$.
Thus, taking into account \cite[Proposition 3.2 and Definition 3.3, p. 208]{Milne}, $\hbox{O}(3)$ is connected. Also by using this algebraic geometry ideas we recognize the fact that 
\begin{equation}\label{dimotres}
\dim\hbox{O}(3)=3.
\end{equation}
Indeed,
the set $S=\{a,b,s\}$ satisfies that  $\K[S]$ is a polynomial algebra and $H$ is finitely generated as a $\K[S]$-module. (see
\cite[16.1 (a)]{Milne}).\end{proof}

Consider the three-dimensional ideal $I=\langle x_1,x_2,x_3\rangle$ of $\L_\K$ (see Proposition \ref{hideput})
and $f\colon \L_\K\to\L_\K$ an automorphism. 
Then $f(I)=I$ and we can consider the induced map $\bar f\colon \L_\K/I\to\L_\K/I$. Since
$\L_\K/I\cong \o(3)$ we may identify these algebras and consider $\bar f\colon\o(3)\to\o(3)$. It is routinary 
to prove that $\bar f\in\aut(\o(3))\cong\hbox{O}(3)$. Thus we have a map $\phi\colon\aut(\L_\K)\to\hbox{O}(3)$
such that $\phi(f)$ is the image of $\bar f$ in $\O(3)$. 
and it is straightforward to see that $\phi$ is a group homomorphism.
We want to prove that $\phi$ is an epimorphism. 

\begin{pr} $\phi$ is an epimorphism.
\end{pr}

\begin{proof}
Consider the $3$-dimensional $\K$-vector space $V:=\K^3$ endowed with the cross-product, that is 
$x\wedge y=(s_2t_3+s_3t_2,s_1t_3+s_3t_1,s_1t_2+s_2t_1)$ where
$x=(s_1,s_2,s_3)$, $y=(t_1,t_2,t_3)$. We will have the ocassion to use also the inner product $\langle x,y\rangle:=s_1t_1+s_2t_2+s_3t_3$.
We know that $V$ is a Lie algebra relative to the product $\wedge$. Furthermore, $V\times V$ is a Lie algebra relative to
$$[(x,y),(z,t)]:=(x\wedge z,x\wedge z+x\wedge t+y\wedge z), \forall x,y,z,t\in V.$$
It is easy to see that $V\times V\cong \L_\K$ where $0\times V$ corresponds to the $3$-dimensional ideal $I$ of $\L_\K$.
More precisely, if we denote by $i,j,k$ the vectors of the canonical basis of $V$, then the isomorphism acts in the form 
\begin{eqnarray}
(i,0)\mapsto b_1,\cr
(j,0)\mapsto b_2,\cr
(k,0)\mapsto b_3,\cr
(0,i)\mapsto x_1,\cr
(0,j)\mapsto x_2,\cr
(0,k)\mapsto x_3.
\end{eqnarray}

So take now an automorphism $g\colon \o(3)\to\o(3)$. We know that relative to the basis $\{b_1,b_2,b_3\}$ of $\o(3)$ such that $[b_i,b_j]=b_k$ (cyclically),
the matrix of $g$ is orthogonal. So the rows of the matrix of $g$ relative to the mentioned basis are three vectors $a_i\in V$, $i=1,2,3$ such that 
$a_i\wedge a_j=a_k$ (cyclically) and $\langle a_i,a_i\rangle =1$ (also $\langle a_i,a_j\rangle=0$ if $i\ne j$).
Now define $f\colon V\times V\to V\times V$ such that 
\begin{eqnarray}
f(i,0)=(a_1,\alpha_1 a_1), &  \text{where } \alpha_1\in \K^\times,\cr
f(j,0)=(a_2,\alpha_2 a_2), &  \text{where } \alpha_2\in \K^\times,\cr
f(k,0)=(a_3,\alpha_3 a_3), &  \text{where } \alpha_3\in \K^\times,
\end{eqnarray}
and the scalars satisfy $\alpha_1+\alpha_2+\alpha_3\ne 1$ (this choice is possible since $\K$ is an infinite field).
Next take $\alpha_0=\alpha_1+\alpha_2+\alpha_3+1\ne 0$ and define 
\begin{eqnarray}
f(0,i)=(0,\alpha_0 a_1), \cr
f(0,j)=(0,\alpha_0 a_2), \cr
f(0,k)=(0,\alpha_0 a_3).
\end{eqnarray}
Now it can be proved that $f$ induces an automorphism of $\L_\K\cong V\times V$ and that $\phi(f)=\bar f=g$.  Thus $\phi$ is an epimorphism.
\end{proof}

Thus $\phi$ is an epimorphism and we define $G:=\ker(\phi)$.
So, we have a short exact sequence 
\begin{equation}\label{ses}
1\to G\to\aut(\L_\K)\to\hbox{O}(3)\to 1
\end{equation}
and we would like to find out more about the group $G$. The elements  $f\in G$
verify that $f(b_i)=b_i+x^{(i)}$ where each $x^{(i)}\in I$ ($i=1,2,3$). If we assume that
$f(x_i)=\sum_j a_{ij}x_j$ and take into account the conditions $[x_i,b_i]=0$ and $[x_i,b_j]=x_k$ (this last
assertion meaning that $i\ne j$ and $k\in\{1,2,3\}\setminus\{i,j\}$), then we find that
$a_{ij}=0$ for $i\ne j$ and $a_{11}=a_{22}=a_{33}$. Thus, the matrix of $f$ in the basis
$\{x_1,x_2,x_3,b_1,b_2,b_3\}$ is of the form
$$\begin{pmatrix}
a_{11} & 0 & 0 & \vline &  &   & \cr
0 & a_{11} & 0 & \vline &   & 0   & \cr
0 & 0 & a_{11} & \vline &  &   & \cr
\hline
 & &  &\vline& 1 & 0 & 0\cr
  &  * & & \vline &0 & 1 & 0\cr
 & &  & \vline &0 & 0 & 1\cr
\end{pmatrix}
.$$
We assume now that $x^{(i)}=\sum_j \lambda_{ij}x_j$.
Then, since  $[b_1,b_2]=b_3$, applying $f$ we have
$[b_1+x^{(1)},b_2+x^{(2)}]=b_3+x^{(3)}$. 
Thus, $[b_1,x^{(2)}]+[b_2,x^{(1)}]=x^{(3)}$ and from here we get 
$$\begin{cases}
\lambda_{31}=\lambda_{13},\cr
\lambda_{32}=\lambda_{23},\cr
\lambda_{33}=\lambda_{11}+\lambda_{22}.
\end{cases}$$
Similarly applying $f$ to $[b_2,b_3]=b_1$ we get
$$\begin{cases}
\lambda_{12}=\lambda_{21},\cr
\lambda_{13}=\lambda_{31},\cr
\lambda_{11}=\lambda_{22}+\lambda_{33},
\end{cases}$$
and so cyclically we conclude that the matrix of $f$ is 
$$\begin{pmatrix}
a_{11} & 0 & 0 & \vline &  &   & \cr
0 & a_{11} & 0 & \vline &   & 0   & \cr
0 & 0 & a_{11} & \vline &  &   & \cr
\hline
 \lambda_{11} & \lambda_{12} & \lambda_{13}  &\vline& 1 & 0 & 0\cr
 \lambda_{12} & \lambda_{22} & \lambda_{23} & \vline &0 & 1 & 0\cr
 \lambda_{13} & \lambda_{23} & \lambda_{11}+\lambda_{22} & \vline &0 & 0 & 1\cr
\end{pmatrix}
.$$
So using $3\times 3$ blocks we see that $G=\left\{\begin{pmatrix}a\cdot 1_3 & 0\cr 
M & 1_3\end{pmatrix}\colon a \in \K^\times, M^t=M, \hbox{tr}(M)=0\right\}$.
Thus, taking into account equation (\ref{dimotres}), since $\dim(G)=6$, from
the existence of the short exact sequence (\ref{ses}), we conclude that  $$\dim(\aut(\L_\K))=9.$$
This fact is in contrast with the results for characteristic $\ne 2$, where as a consequence of
Theorems \ref{complex_case} and \ref{real_case} the dimension of the automorphism group of Lorentz type algebras is $6$.


\begin{theo}
The algebraic group $\aut(\L_\K)$ is connected and $9$-dimensional.
\end{theo}
\begin{proof}
We already know that $\dim \aut(\L_\K)=9$. For the connectedness property we  will apply \cite[Proposition 3.11, p. 210]{Milne} to the above short exact sequence (\ref{ses}). We need to show that $G$ is connected.
But, $G\cong \K^\times\times\mathcal{E}$ where $\mathcal E$ is a vector space (regarded as an affine variety).
Concretely, $\mathcal E$ is the space of traceless symmetric matrices. Thus $\mathcal E$ is connected and so is $G$.
Consequently \cite[Proposition 3.11, p. 210]{Milne} implies the connectedness of $\aut(\L_\K)$. 
\end{proof}

\section{Derivation algebra of $\L_\K$ in characteristic two}
In this section, we will consider the identification of $\L_\K$ with $V\times V$ as in the previous section. 
\begin{theo}\label{nine}
For a field $\K$ of characteristic $2$ we have:
 $$\dim(\der(\L_\K))=12.$$
\end{theo}
\begin{proof}
Let $d\in\der(\L_\K)$ where $\K$ is a field of characteristic $2$, so $d\colon V\times V\to V\times V$. There are four linear maps $\alpha, \beta, \gamma, \delta$ such that $d(x,0)=(\alpha(x),\beta(x))$ and $d(0,x)=(\gamma(x),\delta(x))$. Taking in to account that $d$ is a derivation, we get the following set of identities

\begin{enumerate}
\item\label{pewone} $x\wedge \alpha (y)+\alpha (x)\wedge
   y+\alpha (x\wedge y)+\gamma (x\wedge y)=0$
\item\label{pewtwo} $x\wedge \alpha (y)+\alpha (x)\wedge
   y+x\wedge \beta (y)+\beta (x)\wedge
   y+\beta (x\wedge y)+\delta (x\wedge y)=0$,
\item\label{pewthree} $x\wedge \gamma (y)+\gamma (x\wedge y)=0$,
\item\label{pewfour} $\alpha (x)\wedge y+x\wedge \gamma
   (y)+x\wedge \delta (y)+\delta (x\wedge
   y)=0$,
\item\label{pewfive} $x\wedge \gamma (y)+\gamma (x)\wedge y=0$.
\end{enumerate}

An straightforward computation reveals that the identity \ref{pewfive} implies the existence of a scalar $\lambda \in \K$ such that $\gamma(x)=\lambda x$ for all $x \in V$. As a consequence the identity  \ref{pewthree} is automatically satisfied.
 Now, since we now $\gamma$, we may determine $\alpha$ using equation \ref{pewone}. From this, we get that the matrix of $\alpha$, in canonical  basis, is
 $$\left(
\begin{array}{ccc}
 b_{1,1} & b_{1,2} & b_{1,3} \\
 b_{1,2} & b_{2,2} & b_{2,3} \\
 b_{1,3} & b_{2,3} & \lambda
   +b_{1,1}+b_{2,2}
\end{array}
\right).$$
From the equation \ref{pewfour} we obtain that the matrix of $\delta$ is of the form:
$$\left(
\begin{array}{ccc}
 c_{1,1} & b_{1,2} & b_{1,3} \\
 b_{1,2} & b_{1,1}+b_{2,2}+c_{1,1} &
   b_{2,3} \\
 b_{1,3} & b_{2,3} & \lambda
   +c_{1,1}+b_{2,2}
\end{array}
\right).$$
Finally, equation \ref{pewtwo} gives that the matrix of $\beta$ is of the form:
$$\left(
\begin{array}{ccc}
 f_{1,1} & f_{1,2} & f_{1,3} \\
 f_{1,2} & f_{2,2} & f_{2,3} \\
 f_{1,3} & f_{2,3} & \lambda
   +b_{1,1}+c_{1,1}+f_{1,1}+f_{2,2}
\end{array}
\right).$$
So the free parameters appearing in $d$ are: $$\lambda, b_{11},b_{22},b_{12}, b_{13}, b_{23}, c_{11},f_{11},f_{22},f_{12},f_{13},f_{23},$$ hence $$\dim(\der(\L_\K))=12.$$

\end{proof}

\subsection{Non-smoothness of $\aut(\L_\Phi)$ in characteristic $2$}

Consider now the affine group scheme $\affaut(\L_\Phi)\colon\alg_\Phi\to\grp$, where $\Phi$ is a field of characteristic $2$.
The dimension of this algebraic group is $\dim\aut(\L_\K)$ where $\K$ is the algebraic closure of $\Phi$. As we have proved
in previous sections $\dim\aut(\L_\K)=9$ hence $\dim(\affaut(\L_\Phi))=9$. 
Of course, we have $\Lie(\affaut(\L_\Phi))=\der(\L_\Phi)$ and $\dim(\Lie(\affaut(\L_\Phi)))=\dim(\der(\L_\Phi))=\dim(\der(\L_\K))=12$ 
by Theorem \ref{nine}. Thus we have 
$$\dim(\Lie(\affaut(\L_\Phi)))=12$$ 
and so the affine group scheme $\affaut(\L_\Phi)$ is not smooth.


\section{Structure of the Poincar\'e algebra}


The Poincar\'e group is the inhomogeneous Lorentz group, that is, the group generated by the Lorentz group plus translations.
Similarly the Poincar\'e algebra $\p$ over the field $\K$ is the inhomogeneous Lorentz algebra. We can define it
as the direct sum $\p= \begin{pmatrix}0 & \K^4\cr 0 & \L_\K\end{pmatrix}$ but since
in this section we assume the ground field $\K$ to be algebraically closed and of characteristic other than $2$,
we identify $\L_\K$ with $\o_4(\K)$ allthrough the section (see Lemma \ref{one}). 
Thus, the Poincar\'e algebra $\p$ is the direct sum 
$$\p=\begin{pmatrix}0 & \K^4\cr 0 & \o_4(\K)\end{pmatrix},$$
with the product in this Lie algebra  $\left[\!\begin{pmatrix}0 & v\cr 0 & M\end{pmatrix}\!,\!
\begin{pmatrix}0 & v'\cr 0 & M'\end{pmatrix}\!\right]\!=\!\begin{pmatrix}0 & vM'-v'M\cr 0 & [M,M']\end{pmatrix}$
for any $v,v'\in\K^4$, $M,M'\in\o_4(\K)$. Moreover $ \r=\begin{pmatrix}0 & \K^4\cr 0 & 0\end{pmatrix}$ is an abelian ideal (in fact the radical of $\p$).
The first result we need to prove is:
\begin{lm}
 The ideal $\r$ is minimal.
\end{lm}

\begin{proof}
The group $\O_4(\K)$ of orthogonal matrices acts by conjugation on its Lie algebra $\o_4(\K)$ of skew-symmetric matrices. Also $\O_4(\K)$ acts naturally on $\K^4$ and if two vectors
$v_1,v_2\in\K^4$ are in the same orbit under the action of $\O_4(\K)$ (so $v_2=v_1P$ for some $P\in\O_4(\K)$), then
the ideals of $\p$ generated by $ \begin{pmatrix} 0 & v_i\\0 & 0\end{pmatrix}$, ($i=1.2$) are conjugated by the automorphism $\Omega\in\aut(\p)$ given by 
$$\Omega\begin{pmatrix} 0 & v\\0 & M\end{pmatrix}:=
\begin{pmatrix} 0 & vP\\0 & P^{-1}MP\end{pmatrix}.$$
Consequently, when studying the ideal generated by an element $ x={ \begin{pmatrix}0 & v\cr 0 & 0\end{pmatrix}}$ up to isomorphism, we may replace $v$ with any vector in its orbit under the action of $\O_4(\K)$. To prove that $\r$ is minimal, 
we show that the ideal generated by any nonzero element in $\r$ is $\r$. Thus, 
take 
$ x={ \begin{pmatrix}0 & v\cr 0 & 0\end{pmatrix}}\in\r\setminus\{0\}$.
There are two possibilities to analyze:
\begin{enumerate}
\item Assume first that $v \in \K^4$ is nonisotropic relative to the quadratic form $q(x,y,z,t):=x^2+y^2+z^2+t^2$ of $\K^4$. We may  take $q(v)=1$ since
the ideal generated by $x$ is the same that the ideal generated by any nonzero scalar multiple of $x$. In this case, applying Witt's Theorem, $v$ is in the same orbit as $(1,0,0,0)$ under the action of the orthogonal group. Hence
without loss in generality we may take $v$ to be $(1,0,0,0)$. But then, it is easy to prove that the ideal generated by $x$ is the radical $\r$. 
\item Let us assume now that $v$ is isotropic. Again by Witt's theorem, any two isotropic vectors are in the same orbit under the action of $\O_4(\K)$. So we may take $v$ to be $v=(1,i,0,0)$. In this case, the ideal generated by
$x$ contains all the elements $ \begin{pmatrix}0 & v M\cr 0 & 0\end{pmatrix}$
with $M\in\o_4(\K)$. But the subspace $v M$ is $3$-dimensional hence it contains a nonisotropic vector (the Witt index of $q$ is $2$). Thus the ideal generated by $x$ contains an element $ \begin{pmatrix}0 & w \cr 0 & 0\end{pmatrix}$ with $w$ nonisotropic. Hence this ideal is $\r$.
\end{enumerate}
\end{proof}
 
 We know that $\p/\r\cong\o_4(\K)$ and by Lemma \ref{one} and Proposition \ref{prtwo}, this algebra is a direct sum $\o_4(\K)=I\oplus J$ of two isomorphic ideals $I\cong J\cong\sl_2(\K)$.
 Thus we have seven-dimensional ideals $$\im=\begin{pmatrix}0 & \K^4\cr 0 & I\end{pmatrix},\quad \j=\begin{pmatrix}0 & \K^4\cr 0 & J\end{pmatrix},$$
such that $\im+\j=\p$ and $\im\cap\j=\r$. Furthermore, $\p/\im\cong \sl_2(\K)\cong \p/\j$ so $\im$ and $\j$ are maximal ideals of $\p$.
We may represent this by
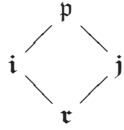
\[
\xygraph{
!{<0cm,0cm>;<1cm,0cm>:<0cm,1cm>::}
!{(0,0)}*+{\p}="p"
!{(-.7,-.7)}*+{\im}="i"
!{(.7,-.7)}*+{\j}="j"
!{(0,-1.4)}*+{\r}="r"
"r"-"i"
"r"-"j"
"i"-"p"
"j"-"p"
}
\]
\begin{lm}\label{zur}
 If $K$ is an ideal of $\p$ such that $[K,\p]\subset\r$, then $K=0$ or $K=\r$.
\end{lm}

\begin{proof}
 If $K\ne 0$ take a nonzero element $x_0=\begin{pmatrix}0 & v_0\cr 0 & M_0\end{pmatrix}\in K$. Then for any
element $x=\begin{pmatrix}0 & v\cr 0 & M\end{pmatrix}\in\p$ we have $[x_0,x]=
\begin{pmatrix}0 & v_0M-vM_0\cr 0 & [M_0,M]\end{pmatrix}\in \r$ hence $M_0$ is in the center of $\o_4(\K)$  which is null.
Thus $x_0\in\r$ and we have proved $K\subset\r$. Finally the minimality of $\r$ implies $K=\r$.
\end{proof}

\begin{lm}
 The unique ideals of $\p$ are: $0$, $\r$, $\im$, $\j$ and $\p$.
\end{lm}

\begin{proof}
Let $K\triangleleft\p$ such that $K\not\in\{0,\r,\im,\j,\p\}$. 
Observe that by maximality of $\im$ we must have either $K\subsetneq\im$ or $K+\im=\p$ (and similarly $K\subsetneq\j$ or $K+\j=\p$).
Thus we have four possibilities:
\begin{enumerate}
 \item $K\subsetneq\im$, $K\subsetneq\j$.
 \item $K\subsetneq\im$, $K+\j=\p$.
 \item $K+\im=\p$, $K\subsetneq\j$.
 \item $K+\im=\p$, $K+\j=\p$.
\end{enumerate}

Since $\r$ is minimal $K\cap\r=\r$ or $K\cap\r=0$. In the first case $\r\subset K$,
and $\r\subset K\cap \im\subset \im$ implying $K\cap\im=\r$ or $K\cap\im=\im$. Similarly $K\cap\j=\r$ or $K\cap\j=\j$ 
\begin{itemize}
 \item If $K\cap\im=\im$, then $\im\subset K$ and by maximality of $\im$ we have $K=\im$ or $K=\p$, a contradiction.
 \item If $K\cap\j=\j$, then we get as above a contradiction.
 \item If $K\cap\im=\r=K\cap\j$, we analize the compatibility of this conditions with (1)-(4). In case that (1) or (2) holds, we have
 $\r\subset K\subsetneq \im$ which implies $K=\r$ a contradiction.
 If (3) holds, then $\r\subset K\subsetneq\j$ implying $K=\r$ a contradiction again. So the situation now is: $K+\im=K+\j=\p$ and $K\cap\im=K\cap\j=\r$.
 But then $[K,\p]\subset[K,\im]+[K,\j]\subset (K\cap\im)+(K\cap\j)\subset\r$ and by Lemma \ref{zur} we get $K=0$ or $K=\r$ a contradiction again.
\end{itemize}
Now we must analyze the case $K\cap \r=0$. If $K\ne 0$, take a nonzero element $x_0=\begin{pmatrix}0 & v_0\cr 0 & M_0\end{pmatrix}\in K$.
Consider now any $v\in K^4$ and  $x=\begin{pmatrix}0 & v\cr 0 & M_0\end{pmatrix}$. Then 
$[x_0,x]= \begin{pmatrix}0 & (v_0-v)M_0\cr 0 & 0\end{pmatrix}\in K\cap\r=0$. So $(v_0-v)M_0=0$ for any $v$ which implies $M_0=0$, a contradiction.
\end{proof}

\begin{co}
 The Poincar\'e algebra $\p$ is centerless. 
\end{co}


\section{Derivation algebra of $\p$}

We analyze in this section the Lie algebra of derivation $\der(\p)$ of the Poincar\'e algebra. The ground field $\K$ is supposed to be
algebraically closed and of characteristic other than $2$. First of all since $\z(\p)=0$ the map $\ad\colon\p\to\der(\p)$ is a monomorphism
and the ideal of inner derivations is $10$-dimensional. We have a linear map $q\colon\der(\p)\to\der(\o_4(\K))$ defined by 
$q(d)=\pi d i$ where $i\colon\o_4(\K)\to\p$ is the natural injection $M\mapsto \begin{pmatrix}0 & 0\cr 0 & M\end{pmatrix}$,
and $\pi\colon\p\to\o_4(\K)$ the map $\begin{pmatrix}0 & v\cr 0 & M\end{pmatrix}\mapsto M$ which is an epimorphism of Lie algebras.
\begin{lm}
 The linear map $q$ is an epimorphism of Lie algebras whose  kernel is isomorphic to the Lie algebra $\K\times\K^4$ with product given
 by $[(\lambda,v_0),(\lambda',v_0')]:=(0,\lambda v_0'-\lambda' v_0)$. Thus $\der(\p)$ fits in a short exact sequence
 $$0\to\K\times\K^4\to\der(\p)\buildrel{q}\over\to\der(\o\nolimits_4(\K))\to 0$$
 and therefore $\dim\der(\p)=11$.
 
\end{lm}
\begin{proof}
The proof that $q$ is a Lie algebra homomorphism is easy if we take into account that the radical $\r$ of $\p$ is $d$-invariant for any derivation $d$ of $\p$ (see \cite{Radical}). Next we prove that $q$ is an epimorphism.
Let $\alpha\in\der(\o_4(\K)$, since the derivations of $\o_4(\K)$ are inner (see Corollary \ref{inner}), there is some $x_0\in\o_4(\K)$ such that $\alpha=\ad(x_0)$.
Then define $d\colon\p\to\p$ such that 
$$d\big[\begin{pmatrix}0 & v\cr 0 & x\end{pmatrix}\big]:=\begin{pmatrix}0 & -vx_0\cr 0 & \alpha(x)\end{pmatrix}.$$
It is easy to check that $d\in\der(\p)$ and $q(d)=\alpha$. In order to determine the kernel of $q$, assume that $q(d)=0$ for a derivation $d$.
Then $d\left[ \begin{pmatrix}0 & 0\cr 0 & x\end{pmatrix}\right]= \begin{pmatrix}0 & \beta(x)\cr 0 & 0\end{pmatrix}$
for some linear map $\beta\colon\o_4(\K)\to\K^4$ which must satisfy the hypothesis on Lemma \ref{etiq}. Therefore there is an element
$v_0\in\K^4$ such that $\beta(x)=v_0x$ for any $x\in\o_4(\K)$. Since any derivation preserves the radical $\r$ of $\p$ (again \cite{Radical}), we also must have
$d\left[\begin{pmatrix}0 & v\cr 0 & 0\end{pmatrix}\right]=\begin{pmatrix}0 & \alpha(v)\cr 0 & 0\end{pmatrix}$
for some linear map $\alpha\colon\K^4\to\K^4$. So there is a matrix $P\in{\mathcal M}_4(\K)$ such that $\alpha(v)=vP$ for any $v$.
Then $d\left[\begin{pmatrix}0 & v\cr 0 & x\end{pmatrix}\right]=\begin{pmatrix}0 & vP+v_0x\cr 0 & 0\end{pmatrix}$ and imposing
the condition that $d$ to be a derivation we get that $P$ must commute with any element in $\o_4(\K)$ but this implies that $P=\lambda 1_4$ where
$\lambda\in\K$ and $1_4$ denotes de identity matrix $4\times 4$. Summarizing:
$d\left[\begin{pmatrix}0 & v\cr 0 & x\end{pmatrix}\right]=\begin{pmatrix}0 & \lambda v+v_0x\cr 0 & 0\end{pmatrix}$
and the scalar $\lambda$ as well as the vector $v_0$ are uniquely determined by $d$. Thus the map
$\ker(q)\to\K\times\K^4$ such that $d\mapsto (\lambda,v_0)$ is a Lie algebras homomorphism where $\K\times\K^4$ is provided
with a Lie algebra structure whose 
product is $[(\lambda,v_0),(\lambda',v_0')]:=(0,\lambda v_0'-\lambda' v_0)$. Thus, the exact sequence of Lie algebras
\begin{equation}\label{ssesone}
0\to\K\times\K^4\to\der(\p)\buildrel{q}\over\to\der(\o\nolimits_4(\K))\to 0
\end{equation}
exists and $\dim\der(\p))=11$. 
\end{proof}

\begin{de}
 For $(\lambda,v_0)\in\K\times\K^4$ define the derivation $d_{\lambda,v_0}\in\ker(q)$ by 
 $$d_{\lambda,v_0}\left[\begin{pmatrix}0 & v\cr 0 & x\end{pmatrix}\right]=\begin{pmatrix}0 & \lambda v+v_0x\cr 0 & 0\end{pmatrix}.$$
\end{de}

\begin{pr}
 The short exact sequence (\ref{ssesone}) is split: there is a monomorphism $$j\colon\der(\o\nolimits_4(\K))\to\der(\p)$$ such that 
for $\alpha\in\der(\o_4(\K))$ one has $j(\alpha)\colon \begin{pmatrix}0 & v\cr 0 & x\end{pmatrix}\mapsto \begin{pmatrix}0 & -vx_0\cr 0 & \alpha(x)\end{pmatrix}$
being $\alpha=\ad(x_0)$. This map satisfies $q j=1$. Consequently any derivation $d\in\der(\p)$ can be uniquely written as a sum $j(\alpha)+d_{\lambda,v}$.
Furthermore:
\begin{enumerate}
  \item $\K\times\K^4\cong\{d_{\lambda,v}\colon (\lambda,v)\in\K\times\K^4\}$ is the radical $\r:=\r(\der(\p))$.
  \item $\der(\p)=\r\oplus j(\der(\o_4(\K)))\cong\r\oplus\sl_2(\K)^2$.
\end{enumerate}
\end{pr}

\begin{proof} It is straightforward to check that $j$ is a monomorphism and that $qj=1$. 
So $\der(\p)=j(\der(\o\nolimits_4(\K)))\oplus \ker(q)$. To finish the proof, it only remains to check that
$\ker(q)\cong\K\times\K^4$ is the radical of $\der(\p)$. By the definition of the product of the Lie algebra $\K\times\K^4$ we see that it is
solvable and so $\ker(q)$ is a solvable ideal and $\der(\p)/\ker(q)\cong\sl_2(\K)^2$ is semisimple (take into account Lemmas \ref{one} and \ref{carras}). Thus, $\ker(q)$ is the radical of $\der(\p)$.
\end{proof}


\section{Automorphism group of $\p$}


The ground field $\K$ is supposed to be
algebraically closed and of characteristic other than $2$. 
Consider the group homomorphism $Q\colon\aut(\p)\to\aut(\o_4(\K))$ such that $Q(f)=\pi f i$ where $\pi$ and $i$ have been defined at the
beginning of the previous section. Given $\theta\in\aut(\o_4(\K))$ there is an $x_0\in\GL_4(\K)$ such that $\theta=\Ad(x_0)$.
The map $f\colon\p\to\p$ such that $
f\colon \begin{pmatrix}0 & v\cr 0 & x\end{pmatrix}\mapsto \begin{pmatrix}0 & vx_0^{-1}\cr 0 & \theta(x)\end{pmatrix}$
is an automorphism of $\p$ and $Q(f)=\theta$. Therefore $Q$ is an epimorphism whose  kernel is described in the following
\begin{lm}
The kernel $\ker(Q)$ is the subgroup of
automorphisms $f_{\lambda,v_0}$ where $\lambda\in\K^\times$ and $v_0\in\K^4$; such that $$f_{\lambda,v_0}\colon \begin{pmatrix}0 & v\cr 0 & x\end{pmatrix}\mapsto
\begin{pmatrix}0 & \lambda v+v_0x\cr 0 & x\end{pmatrix}.$$ 
\end{lm}
\begin{proof} Let $f\in\aut(\p)$, since $\r$ is $f$-invariant, we have
$$f\left(\left(
\begin{array}{cc}
 0 & v \\
 0 & x
\end{array}
\right)\right)=\left(
\begin{array}{cc}
 0 & \mu (v)+\alpha (x) \\
 0 & \beta (x)
\end{array}
\right)$$ for some linear maps $\mu\colon\K^4\to\K^4$, 
$\alpha\colon\o_4(\K)\to\K^4$, $\beta\colon\o_4(\K)\to\o_4(\K)$. If we assume that $f\in\ker(Q)$ then $\pi f i=1$ and this implies $\beta=1$. Now imposing the condition that $f$ is a Lie algebra homomorphism, we get:
\begin{enumerate}
\item $\alpha([x,y])=\a(x)y-\a(y)x$ for any $x,y\in \o_4(\K)$.
\item $\mu(v x)=\mu(v) x$ for any $v\in\K^4$ and $x\in\o_4(\K)$.
\end{enumerate}
Applying Lemma \ref{etiq} to $\a$, there is some $v_0\in\K^4$ such that
$\a(x)=v x$ for any $x$. Also, since $\mu$ is linear there must be a matrix $P$ such that $\mu(v)=vP$ for any $v$. But imposing the condition in the second intem above we find that $P$ must be a scalar multiple  of the identity. So
$\mu(v)=\lambda v$ for any $v\in\K^4$. Finally the fact that $f$ is injective implies $\lambda\ne 0$.
\end{proof}

Thus $\ker(Q)\cong\K^\times\times\K^4$ with multiplication
$(\lambda,v_0)(\mu,w_0):=(\lambda\mu,\lambda w_0+v_0)$. Consequently $\aut(\p)$ fits in a exact sequence
\begin{equation}\label{ssestwo}
 1\to\K^\times\times\K^4\to\aut(\p)\buildrel{Q}\over\to\aut(\o\nolimits_4(\K))\to 1
 \end{equation}
which is also split: the map $J\colon\aut(\o_4(\K)) \to\aut(\p)$ such that $ J(\theta)\colon  
\begin{pmatrix}0 & v\cr 0 & x\end{pmatrix}\mapsto 
\begin{pmatrix}0 & vx_0^{-1}\cr 0 & \theta(x)\end{pmatrix}$
for any automorphism $\theta=\Ad(x_0)$ of $\o_4(\K)$, is a monomorphism and $QJ=1$.  As a corollary of this, we have $$\dim\aut(\p)=11$$
since $\dim\aut(\o_4(\K))=6$ (see  Lemma \ref{one} and formula (\ref{bmw})).

\begin{pr} The automorphism group $\aut(\p)$ agrees with the group
$$\aut(\p)=\{f_{\lambda,v_0,x_0}\colon \lambda\in\K^\times, v_0\in\K^4,\Ad(x_0)\in\aut(\o\nolimits_4(\K)\}$$ where
$$f_{\lambda,v_0,x_0}
\left[\begin{pmatrix} 0 & v \cr 0 &x\end{pmatrix}\right]:=
\begin{pmatrix} 0 & \lambda vx_0^{-1}+v_0x_0xx_0^{-1} \cr 0 & x_0xx_0^{-1}\end{pmatrix}.$$
Some multiplicative relations of these elements are: 
$$f_{\lambda,v_0,x_0} f_{\mu,v_1,x_1}=f_{\lambda\mu,\lambda v_1x_0^{-1}+v_0,x_0x_1},$$
$$f_{\lambda,v_0,x_0}^{-1}=f_{\lambda^{-1},-\lambda^{-1}v_0x_0,x_0^{-1}.}$$
\end{pr}
\begin{proof}
Since the exact sequence (\ref{ssestwo}) is split, any automorphism of $\p$
is of the form $f_{\lambda,v_0} J(\theta)$ and this gives the description of $\aut(\p)$ claimed in the statement of the Proposition. The multiplicative relations are straightforward to check.
\end{proof}
\section*{Acknowledgments}
All the  authors have been partially supported by the Spanish Ministerio de Econom\'{\i}a y Competitividad and Fondos FEDER, jointly, through project MTM2013-41208-P,  by the Junta de Andaluc\'{\i}a and Fondos FEDER, jointly, through projects FQM-336 and FQM-7156. 


\end{document}